\documentclass{custom1}
\usepackage[mathscr]{eucal}
\usepackage{hyperref}
\usepackage{amssymb,amsmath}
\usepackage{graphicx}

\newtheorem{theorem}{Theorem}[section]
\newtheorem{lemma}[theorem]{Lemma}

\newtheorem{proposition}[theorem]{Proposition}
\newtheorem{corollary}[theorem]{Corollary}

\theoremstyle{definition}

\newtheorem*{claim}{Claim}
\newtheorem*{claim1}{Claim 1}
\newtheorem*{claim2}{Claim 2}
\newtheorem*{claim3}{Claim 3}

\theoremstyle{remark}
\newtheorem*{remark}{Remark}
\newtheorem*{remarks}{Remarks}
\newtheorem*{observation}{Observation}

\numberwithin{equation}{section}

\title[An inverse theorem]{An inverse theorem: when the measure of the sumset is the sum of the measures in a locally compact abelian group}
\author{John T. Griesmer}
\address{Department of Mathematics\\
The University of British Columbia\\
Room 121, 1984 Mathematics Road\\
Vancouver, B.C.\\
Canada V6T 1Z2 }
\curraddr{Department of Mathematics, University of Denver\\
John Greene Hall, Room 203
2360 S. Gaylord St.
Denver, CO 80208}
\email{John.Griesmer@du.edu}

\begin{document}
\begin{abstract}
We classify the pairs  of subsets $A$, $B$  of a locally compact abelian group $G$ satisfying $m_*(A+B)=m(A)+m(B)$, where $m$ is the Haar measure for $G$ and $m_*$ is inner Haar measure.  This generalizes M. Kneser's classification of  such pairs when $G$ is assumed to be connected.  Recently, D.~Grynkiewicz classified the pairs of sets $A$, $B$ satisfying $|A+B|=|A|+|B|$ in an abelian group, and our result is complementary to that classification.  Our proofs combine arguments of Kneser and Grynkiewicz.
\end{abstract}

\maketitle

\section{Introduction}

\subsection{Small sumsets in LCA groups}    Given subsets $A$ and $B$ of an abelian group $G$, we consider their \textit{sumset} $A+B:=\{a+b: a\in A, b\in B\}$. A general theme in additive combinatorics is that $A$ and $B$ must be highly structured when $A+B$ is not very large in comparison to $A$ and $B$; see \cite{GrynkiewiczKemperman}, \cite{Nathansoninverse}, or \cite{TaoVu} for many instances of this theme.  We consider the case where $G$ is a locally compact abelian (LCA) group,  extending the investigations of \cite{Kneser56}.   Theorem 1 of \cite{Kneser56} describes the pairs $(A,B)$ of Haar measurable subsets of an LCA group $G$ with Haar measure $m$ satisfying $m_*(A+B)<m(A)+m(B)$; such pairs are often called \textit{critical pairs}. Here $m_*$ is the inner measure corresponding to $m$, so $m_*(S)=\sup\{m(E): E\subseteq S, E \text{ is compact}\}$.  If $S\subseteq G$ and $t\in G$, we write $S+t$ for the set $\{s+t:s\in S\}$, and write $C\sim D$ if $m(C\triangle D)=0$.

\begin{proposition}[\cite{Kneser56}, Theorem 1]\label{Ktopinequality}  Let $G$ be a locally compact abelian group with Haar measure $m$,  and $A$, $B\subseteq G$ measurable sets with $m_*(A+B)<m(A)+m(B)$.  Then the group $H:=\{t\in G: A+B+t \sim A+B\}$ is compact and open,
\begin{align}
A+B=A+B+H,
\end{align}  and
\begin{align}\label{critical}
m(A+B)=m(A+H)+m(B+H)-m(H).
\end{align}
\end{proposition}
Proposition \ref{Ktopinequality}, especially equation (\ref{critical}), plays a crucial role in our results.

\smallskip

Since connected groups have no proper compact open subgroups, Proposition \ref{Ktopinequality} implies $m_*(A+B)\geq \min\{m(A)+m(B),m(G)\}$ whenever $G$ is connected, generalizing results Raikov \cite{Raikov},  Macbeath \cite{Macbeath}, and Shields \cite{Shields}.  The following theorem from \cite{Kneser56} completely classifies the pairs of subsets of a compact, \textit{connected}, abelian group satisfying $m_*(A+B)=m(A)+m(B)$, when $m(A)>0$ and $m(B)>0$.  We call such a pair $(A,B)$ a \textit{sur-critical} pair.  Here $\mathbb T$ is the group $\mathbb R/\mathbb Z$; see \S \ref{intervalsdef} for our usage of the term ``interval."

\begin{proposition}[\cite{Kneser56}, Theorem 3]\label{Kconnected}  Let $G$ be a compact connected abelian group with Haar measure $m$, and $A$, $B\subseteq G$  measurable sets satisfying $m(A)>0$, $m(B)>0$, and $m_*(A+B)=m(A)+m(B)<1$.  Then there is a continuous surjective homomorphism $\chi:G\to \mathbb T$, and there are intervals $I$, $J\subseteq \mathbb T$ such that $A\subseteq \chi^{-1}(I)$, $B\subseteq \chi^{-1}(J)$,  $m(A)=m(\chi^{-1}(I))$, and $m(B)=m(\chi^{-1}(J))$.
\end{proposition}

Kemperman \cite{KempermanTopological} generalized part of Proposition \ref{Ktopinequality} to the case where $G$ is not abelian, showing in particular that $m_*(A\cdot B)\geq m(A)+m(B)$ for subsets of a general locally compact connected group (here $A\cdot B$ denotes $\{ab:a\in A, b\in B\}$).  Bilu \cite{Bilu} investigates the inequality $m(A+B)<m(A)+m(B)+\min\{m(A),m(B)\}$ when the ambient group is $\mathbb T^d$ for some $d$; the case $d=1$ is studied in \cite{MoskvinFreimanJudin}.

\subsection{New results}  When $G$ is disconnected, the pairs $(A,B)$ of subsets of $G$ satisfying $m_*(A+B)=m(A)+m(B)$ are not completely described by Proposition \ref{Kconnected}.   For example, take $G=(\mathbb Z/5\mathbb Z)\times \mathbb T$, and let $A=(\{0\}\times \mathbb T) \cup (\{1\} \times [0, 1/2])$, $B=(\{0,1\}\times \mathbb T) \cup (\{2\} \times [0,1/4])$, so that $A+B=(\{0,1,2\}\times \mathbb T) \cup (\{1\} \times [0,3/4])$.  Other examples are discussed in \S \ref{examples}. We classify the pairs $(A,B)$ of measurable subsets of a locally compact abelian group satisfying $m(A)>0$, $m(B)>0$, and $m_*(A+B)=m(A)+m(B)$, generalizing Proposition \ref{Kconnected} to the case where $G$ may be disconnected.   Our main result is Theorem \ref{main}, where $G$ is assumed to be compact.   See \S \ref{terminology} for notation and terminology.  In \S \ref{lc} we discuss the case where $G$ is not compact.

\begin{theorem}\label{main}  Let $G$ be a compact abelian group with Haar measure $m$, and let $A$, $B\subseteq G$ be measurable sets such that $m(A)>0$, $m(B)>0$, and $m_*(A+B)=m(A)+m(B)$.  Then at least one of the following is true\textup:
\begin{enumerate}
\item[(P)]  There is a compact open subgroup $K\leqslant G$ with $A+K\sim A$ and  $B+K\sim B$.

\smallskip

\item[(E)]  There are measurable sets $A'\supseteq A$ and  $B'\supseteq B$ such that
\begin{align*}
m(A')+m(B')>m(A)+m(B),
\end{align*}  and $m_*(A'+B')=m_*(A+B)$.

\smallskip

\item[(K)]  There is a compact open subgroup $K\leqslant G$, a continuous surjective homomorphism $\chi:K\to \mathbb T$, intervals $I$, $J\subseteq \mathbb T$, and elements $a$, $b\in G$ such that $A\subseteq a+\chi^{-1}(I)$, $B\subseteq b+\chi^{-1}(J)$, $m(A)=m(\chi^{-1}(I))$, and $m(B)=m(\chi^{-1}(J))$.

\smallskip

\item[(QP)]  There is a compact open subgroup $K\leqslant G$ and partitions $A=A_1\cup A_0$, $B=B_1\cup B_0$ such that $A_0\neq \varnothing$, $B_0\neq \varnothing$, at least one of $A_1\neq \varnothing$, $B_1\neq \varnothing$, and

\begin{enumerate}
\item[(QP.1)]
$A_1+K\sim A_1$, $B_1+K\sim B_1$,  while  $A_0$ and $B_0$ are each contained in a coset of $K$ and $(A_1+K)\cap A_0=(B_1+`K)\cap B_0=\varnothing$\textup{;}

\smallskip

\item[(QP.2)]  $A_0+B_0+K$ is a unique expression element of $A+B+K$ in $G/K$\textup{;}

\smallskip

\item[(QP.3)]    $m_*(A_0+B_0)=m(A_0)+m(B_0)$.
\end{enumerate}
\end{enumerate}
\end{theorem}
\begin{remarks} (i) Corollary \ref{QPfour} provides more detail in conclusion (QP); conclusion (E) is examined in \S \ref{Eremark}.  Corollary \ref{roman} is a convenient rephrasing of Theorem \ref{main}.

\smallskip

\noindent (ii)  We do not provide explicit constructions of all pairs satisfying conclusion (QP), so the statement of Theorem \ref{main} is not a complete characterization of the sur-critical pairs for a compact abelian group.

\smallskip

\noindent (iii)  The labels (P), (E), (K), and (QP) stand for ``periodic," ``extendible," ``Kneser," and ``quasi-periodic," respectively.  See \S \ref{setsandpairs} for elaboration.

\smallskip

\noindent (iv) Theorem \ref{main} was developed partly to help answer Question 4.1 of \cite{Jin10}. \hfill $\blacksquare$ \end{remarks}

Theorem \ref{main} is proved in \S \ref{proof}.  A sequence of lemmas in \S \ref{lemmas} reduces the proof to a special case.  We then apply the $e$-transform  (\S \ref{etsection}) as in the proof of Proposition \ref{Kconnected} from \cite{Kneser56}.  Lemma \ref{transformcase1} handles difficulties related to applying the $e$-transform in a disconnected group; its statement and proof are inspired by arguments from \cite{Grynkiewicz}.

\smallskip

Some easy consequences of Theorem \ref{main} are discussed in \S \ref{corollarysection}.  Examples are given in \S \ref{examples} showing that each conclusion in Theorem \ref{main} may occur, and none can be omitted.  \S \ref{corollarysection} and \S \ref{examples} are technically irrelevant to the proof of Theorem \ref{main}.

\subsection{When one of $m(A)=0$ or $m(B)=0$.}  When $m(B)=0$, $0<m(A)<\infty$, and $B$ contains the identity element of $G$, Theorem 4 of \cite{Kneser56} says that $m_*(A+B)=m(A)+m(B)$ if and only if the closed subgroup $G'\leqslant G$ generated by $B$ is compact, and $A$ can be partitioned as $A=C\cup D$, where $m(C)=0$, $m_*(B+C)=0$ and $m(G'+D)=m(D)$.  We have nothing to say regarding the case where $m(A)=m(B)=0$.

\subsection{Context}  The conclusion of Theorem \ref{main} is trivial for discrete groups $G$, where the trivial subgroup $\{0\}$ is compact and open, so conclusion (P) holds with $K=\{0\}$.  The article \cite{Grynkiewicz} classifies the sur-critical pairs for discrete groups, and therefore provides a classification of the pairs $(A,B)$ satisfying conclusion (P) in Theorem \ref{main} (see \S \ref{periodicexamples}).   We summarize the development of some related results.

\smallskip

\textit{Inverse theorems} in additive combinatorics deduce properties of subsets $A$, $B$ of an abelian group from a hypothesis on the sumset $A+B$, while   \textit{direct theorems} deduce properties of $A+B$ from hypotheses on $A$ and $B$.  One of the earliest direct theorems is the Cauchy-Davenport inequality, which states that $|A+B|\geq \min\{|A|+|B|-1, p\}$ whenever $A$, $B\subseteq \mathbb Z/p\mathbb Z$ for some prime $p$; here $|S|$ denotes the cardinality of the set $S$.  The corresponding inverse theorem, due to Vosper \cite{Vosper1,Vosper2}, classifies the pairs $(A,B)$  of subsets of $G=\mathbb Z/p\mathbb Z$ satisfying $|A+B|=|A|+|B|-1$, when $|A|+|B|< p$: the equation holds if and only if $A$ and $B$ are arithmetic progressions with the the same common difference, or one of $|A|=1$, $|B|=1$.  When $|A|+|B|=p$, an elementary analysis shows that the equation $|A+B|=|A|+|B|-1$ is satisfied exactly when $B$ is a translate of $-(G\setminus A)$, and when $|A|+|B|=p+1$, the equation $|A+B|=|A|+|B|-1$ is always satisfied if $A$ and $B$ are nonempty.

\smallskip

Generalizing Vosper's theorem and strengthening  Proposition \ref{Ktopinequality}, Kemperman \cite{Kemperman60} classified the pairs of finite subsets $A$, $B$ of an arbitrary abelian group satisfying $|A+B|<|A|+|B|$;  see \cite{GrynkiewiczKemperman} for exposition.  The articles \cite{HaRo,HaSeZe} classify pairs $(A,B)$ satisfying $|A+B|=|A|+|B|$ in certain ambient abelian groups.  More recently, Grynkiewicz \cite{Grynkiewicz} classified the pairs $(A,B)$ of finite subsets of an arbitrary abelian group satisfying $|A+B|=|A|+|B|$.  That classification is somewhat intricate, so we do not reproduce it here.   Concatenating Theorem \ref{main} with the results of \cite{Kemperman60} and \cite{Grynkiewicz} yields a very precise description of pairs $(A,B)$ of subsets of a compact abelian group satisfying $m_*(A+B)=m(A)+m(B)$, but we will not state this description explicitly.

\section{Terminology, notation, and background}\label{terminology} We assume knowledge of the theory of locally compact abelian groups.   The books \cite{HewittandRoss1} and \cite{ReiterStegeman} together provide nearly sufficient background, as does \cite{Kneser56} and its bibliography.    We need Theorem A of \cite{Mueller65}, which is proved in \cite{Mueller62}.

\smallskip

Some of our terminology and notation is taken from \cite{Grynkiewicz} and \cite{Kneser56}; when $G$ is discrete, some of our definitions coincide with those from \cite{Grynkiewicz}.

\subsection{General conventions}  Throughout, $G$ will denote a locally compact abelian group and $m_G$ will be its Haar measure.  We make the standard assumption that $G$ satisfies the $T_0$ separation axiom.  When there is no chance of confusion, we write $m$ for $m_G$, and $m_*$ for the corresponding inner measure.   The term \textit{measurable} in reference to a subset of $G$ will always mean ``lies in the completion of the Borel $\sigma$-algebra with respect to $m$."  We will exploit inner regularity of Haar measure for compact groups: $m(C)=\sup\{m(E): E\subseteq C, E \text{ is compact}\}$ when $C$ is measurable and $G$ is compact.  Haar measure will always be normalized for compact groups, so $m_G(G)=1$ for such $G$.

\smallskip

If $S\subseteq G$, $S^c$ will denote the set theoretic complement $G\setminus S$.  The symbol $-S$ denotes $\{-s:s\in S\}$, while $A-B$ denotes the difference set $\{a-b:a\in A, b\in B\}$.

\smallskip

 The identity element of $G$ will be written as $0_G$, or simply $0$ if there is no ambiguity.

\subsubsection{Intervals}\label{intervalsdef}  The symbol $\mathbb T$ will denote the group $\mathbb R/\mathbb Z$ with its usual topology,  and  ``an interval in $\mathbb T$" means a set of the form $[x,y]+\mathbb Z$, where $x\leq y <x+1$.  We omit the term ``closed" as we will never refer to non-closed intervals.  We frequently exploit the following property of intervals:  if $I$, $J\subseteq \mathbb T$ are intervals with $m_{\mathbb T}(I)+m_{\mathbb T}(J)<1$, and $x\notin I$, then $m_{\mathbb T}((x+J)\setminus (I+J))>0$.  Consequently, if $\chi:G\to \mathbb T$ is a continuous surjective homomorphism and $t\notin \chi^{-1}(I)$, then $m_G((t+\tilde{J})\setminus (\tilde{I}+\tilde{J}))>0$, where $\tilde{I}=\chi^{-1}(I)$ and $\tilde{J}=\chi^{-1}(J)$.

\subsubsection{Subgroups and quotients} If $K\leqslant G$ is a closed subgroup, the quotient $G/K$ with the quotient topology is a locally compact abelian group.  We may identify subsets of $G$ of the form $A+K$ with subsets of $G/K$, and conversely subsets of $G/K$ may be identified with subsets of $G$.

\smallskip

A \textit{$K$-coset decomposition} of a set $A\subseteq G$ is the collection of sets $A_i=A\cap K_i$, where $K_i$ ranges over the cosets of $K$ having nonempty intersection with $A$.

\smallskip

We use without comment the following well-known facts:

\begin{itemize}
\item If $K\leqslant G$ is a measurable subgroup of a compact group, then $m(K)>0$ if and only if $K$ is open, if and only if $K$ has finite index in $G$.  The index of $K$ in $G$ is $1/m(K)$.

\item If $K\leqslant G$ is an open subgroup, the group $G/K$ is discrete.
\end{itemize}

\subsubsection{Disintegration of Haar measure}   This subsection regards technicalities occurring only in \S \ref{trivializing}.  For simplicity we assume $G$ is compact.  Given a closed subgroup $K\leqslant G$, we consider the Haar measure $m_K$ as a measure on $G$.   If $A\subseteq G$, and $x\in G$, we consider $m_K(A-x)$, which may be regarded as the $m_K$-measure of $A$ in the coset $x+K$ of $K$.  Although the function $g(x):= m_K(A-x)$ depends only on the coset $K+x$, and may therefore be regarded as a function whose domain is (a subset of) $G/K$, we prefer to regard $g$ as a function whose domain is (a subset of) $G$.  The measure $m_G(A)$ can be recovered in a natural way from the measures $m_K(A-x)$; in other words, Haar measure can be disintegrated over the cosets of a closed subgroup.  This is the content of the following proposition; cf.~\S 2 of \cite{Kneser56}. It may be obtained by specializing Theorem 3.4.6 of \cite{ReiterStegeman} to the case where $G$ is compact and abelian.

\begin{proposition}\label{Weil}  Let $K\leqslant G$ be a closed subgroup of $G$ with Haar measure $m_K$, and let $f\in L^1(m)$ be a real-valued function.  Then the function $\tilde f: G\to \mathbb R$ given by $\tilde{f}(x)= \int f(x+t)\, dm_K(t)$ is defined for $m$-almost every $x$, $\tilde{f}$ is $m_G$-measurable, and $\int \tilde f\, dm=\int f\, dm$.  In particular, if $A\subseteq G$ is measurable, then the function $x\mapsto m_K(A-x)$ is measurable, and $\int m_K(A-x)\, dm(x) = m(A)$. \hfill $\blacksquare$
\end{proposition}

\subsubsection{Unique expression elements}\label{uees} We say $c\in A+B$ is a \textit{unique expression element of} $A+B$ if $c=a_0+b_0$ for some $a_0\in A$, $b_0\in B$, and $a+b=c$ implies $a=a_0$ and $ b=b_0$  when  $a\in A$ and $b\in B$.  Unique expression elements play an important role in \cite{Kemperman60}, the classification of pairs $(A,B)$ satisfying $|A+B|=|A|+|B|-1$; Corollary \ref{QPfour} connects that classification to conclusion (QP) of Theorem \ref{main}.  If $K\leqslant G$ is a subgroup, the phrase ``$C+K$ is a unique expression element of $A+B+K$ in $G/K$" means that $C+K= a_0+b_0+K$ for some $a_0\in A$, $b_0\in B$, and whenever $a\in A$, $b\in B$, and $a+b+K=C+K$, then $a+K=a_0+K$ and $b+K=b_0+K$.

\subsubsection{Critical and sur-critical pairs}\label{critdef}  If $A$, $B\subseteq G$ are measurable sets satisfying $m_*(A+B)=m(A)+m(B)$, we call $(A,B)$ a \textit{sur-critical pair} for $G$.  We call a pair $(A,B)$ satisfying $m_*(A+B)<m(A)+m(B)$ a \textit{critical pair} for $G$.

\subsubsection{Similarity and the essential stabilizer}\label{Hdef}  If $C$, $D\subseteq G$, we write $C\sim D$ to mean $m(C \triangle D)=0$, and  we write $C\subset_m D$ to mean $m(C\setminus D)=0$.

\smallskip

When $C$ is measurable and $m(C)<\infty$, the group $H(C):=\{t\in G: C+t\sim C\}$ is a compact subgroup of $G$ (\cite{Kneser56}, Lemma 4); we may refer to $H(C)$ as the \textit{essential stabilizer} of $C$.  The group $H(A+B)$ plays an important role in \cite{Kneser56} and in the proof of Theorem \ref{main}.  Proposition \ref{Ktopinequality} says that if $m_*(A+B)<m(A)+m(B)$ then $H(A+B)$ is compact and open, and $A+B+H(A+B)=A+B$.  The equation $C=C+H(C)$ may fail in general, and more importantly, the similarity $C\sim C+H(C)$ may fail.  For example, take $G= (\mathbb Z/5\mathbb Z) \times \mathbb T$, and let $C= (\{0\}\times \mathbb T) \cup \{(1,0)\}$.  Then $H(C)=\{0\}\times \mathbb T$, so $m(C+H(C))= 2/5$, while $m(C)=1/5$.

\subsection{Special sets and special pairs}\label{setsandpairs}

\subsubsection{Periodicity}\label{periodicdef}  If $C\sim C+K$ for some some compact open subgroup $K\leqslant G$,  we call $C$ \textit{periodic} with \textit{period} $K$.  Otherwise, we call $C$ \textit{aperiodic}.  Note that when $G$ is not discrete, the assertion $H(C)=\{0\}$ implies $C$ is aperiodic, but an aperiodic $C$ may have $m(H(C))>0$.   We will exploit the following  relation between the essential stabilizer and periodicity. This observation is a consequence of Lemma \ref{H(S)}, although it may be obtained by more elementary means.

\begin{observation} If $m(H(C))>0$, then for all cosets $H_i$ of $H(C)$, either $m(C\cap H_i)=0$ or $m(C\cap H_i)=m(H_i)$.  Consequently, if $m(H(C))>0$, there is a measurable subset $C'\subseteq C$ such that $C\sim C'\sim C'+H(C)$.  \hfill $\blacksquare$
\end{observation}

\begin{remark}  When $G$ is discrete, every subset of $G$ is periodic according to our definition.  This differs from the terminology of \cite{Grynkiewicz}, where a periodic set $S$ must satisfy $S+t \sim S$ for some $t\neq 0$.  \hfill $\blacksquare$
\end{remark}

\subsubsection{Extendibility and nonextendibility}   Let $A$ and $B$ be measurable subsets of a compact abelian group $G$.  We say that $A$ is \textit{extendible with respect to} $B$ if there is a measurable set $A'\supseteq A$ with $m(A')>m(A)$ and $m_*(A'+B)=m_*(A+B)$.  We say that the pair $(A,B)$ is \textit{extendible} if $A$ is extendible with respect to $B$ or $B$ is extendible with respect to $A$.  Otherwise, we say that $(A,B)$ is \textit{nonextendible}.  The nonextendibility of a pair $(A,B)$ may be expressed as follows:  if $A'\supseteq A$ and $B'\supseteq B$ are measurable and $m_*(A'+B')=m_*(A+B)$, then $A'\sim A$ and $B'\sim B$.

\smallskip

 When $m_*(A+B)=m(A)+m(B)$ and $(A,B)$ is extendible, Proposition \ref{Ktopinequality} implies $H(A+B)$ is compact and open, and $A+B\sim A+B+H(A+B)$.  Consequently, $A+B$ is measurable when $(A,B)$ is extendible.  The following example lemma  illustrates how nonextendibility will be exploited in subsequent proofs.

 \begin{lemma}  If $K\leqslant G$ is a compact open subgroup, $A+B\sim A+B+K$, and $(A,B)$ is nonextendible, then $A+K\sim A$.
 \end{lemma}

\begin{proof} The similarity $A+B\sim A+B+K$ can be rewritten as $(A+K)+B\sim A+B$.  Since $A+K\supseteq A$ and $(A,B)$ is nonextendible, we have $A+K\sim A$. \end{proof}

In subsequent proofs, such as those of Lemmas \ref{reduciblelemma1}, \ref{reduciblelemma2}, \ref{QP3}, and \ref{transformcase1},   we will omit the above reasoning.

\subsubsection{Complementary pairs}\label{complementarydef}  If $G$ is compact and $m(A+B)=m(A)+m(B)=1$, call $(A,B)$ a \textit{complementary} pair.  When $G$ is infinite, it is easy to construct such pairs $(A,B)$ with $A+B\neq G$: let $A\subseteq G$ be any measurable set meeting every coset of every finite index subgroup of $G$ with $0<m(A)<1$, and let $B\subseteq G$ have $m(B)=1-m(A)$.  Then $(A,B)$ is a complementary pair by Proposition \ref{Ktopinequality}.  In particular, if $G$ is connected and $A$ is any measurable subset of $G$, then $(A,-A^c)$ is a complementary pair.

\smallskip

Complementary pairs $(A,B)$ with $m(A)>0$ and $m(B)>0$ satisfying $A+B\neq G$ can be described as follows.  If $(A,B)$ is a complementary pair, and $A+B\neq G$, then $A\cap (t-B)=\varnothing$ for some $t\in G$, so $t-B\sim A^c$.  If $s\notin H(t-B)$ ($=H(A^c)=H(A)$), then $m(A\cap (t+s-B))>0$, so $t+s\in A+B$.  It follows that $A+B$ contains a translate of $G\setminus H(A)$.

\smallskip

If $K\leqslant G$ is a compact open subgroup, $A$ and $B$ are each contained in a coset of $K$, and $m(A+B)=m(A)+m(B)=m(K)$, we say that $(A,B)$ is \textit{complementary with respect to} $K$.  Such pairs form a subclass of the extendible pairs.

\subsubsection{Reducibility}\label{reducibledef}  If there are measurable subsets $A'\subseteq A$ and $B'\subseteq B$ such that $m(A')=m(A)$, $m(B')=m(B)$, and $m_*(A'+B')<m_*(A+B)$, we say that $(A,B)$ is \textit{reducible}.  Otherwise, we say $(A,B)$ is \textit{irreducible}.

\subsubsection{Quasi-periodicity}\label{qpdef}  Let $K\leqslant G$ be a compact open subgroup. A set $A\subseteq G$ is called \textit{quasi-periodic with respect to $K$} if $A$ can be partitioned into two nonempty sets $A_1$ and  $A_0$ such that $(A_1+K)\cap (A_0+K)=\varnothing$, $A_1\sim A_1+K$, and $A_0$ is contained in a coset of $K$.  We call $A_1\cup A_0$ a \textit{quasi-periodic decomposition} of $A$, and we say $K$ is a \textit{quasi-period} of $A$.  Note that $A$ may have more than one quasi-period.

\smallskip

Call a pair $(A,B)$ of subsets of $G$  \textit{quasi-periodic with respect to $K$} if one of $A$ or $B$ is quasi-periodic with respect to $K$ and the other is either contained in a coset of $K$ or is quasi-periodic with respect to $K$.  This means that $A=A_1\cup A_0$, $B=B_1\cup B_0$, $A_1\sim A_1+K$, $B_1\sim B_1+K$, $A_0$ and $B_0$ are nonempty and contained in cosets of $K$, $(A_1+K)\cap A_0=(B_1+K)\cap B_0=\varnothing$, and at least one of $A_1$ and $B_1$ is nonempty.  We say $A_1\cup A_0$, $B_1\cup B_0$ is a \textit{quasi-periodic decomposition of} $(A,B)$, and $K$ is a \textit{quasi-period} of $(A,B)$.  Note that if $(A,B)$ is quasi-periodic with respect to $K$, then $A+B$ is quasi-periodic with respect to $K$.

\begin{remark}  Our insistence that $A_1$ be nonempty for $A$ to be quasi-periodic differs from the convention of \cite{Grynkiewicz}; the analogous definition in \cite{Grynkiewicz} allows $A_1=\varnothing$. \hfill $\blacksquare$
\end{remark}

\subsubsection{Repeated decompositions}\label{repeateddef}  If $A_1\cup A_0$ is a quasi-periodic decomposition of $A$ with quasi-period $K_0$ and $A_0= C_1\cup C_0$ is a quasi-periodic decomposition of $A_0$ with quasi-period $K_1\leqslant K_0$, then setting $A_0'=C_0$ and $A_1'=A\setminus C_0$, the partition $A=A_1'\cup A_0'$ is a quasi-periodic decomposition of $A$ with respect to $K_1$.  Consequently, if $(A,B)$ satisfies (QP) of Theorem \ref{main} with $A=A_1\cup A_0$, $B=B_1\cup B_0$, and subgroup $K=K_1\leqslant G$, and $(A_0,B_0)$ also satisfies (QP) with subgroup $K=K_2\leqslant G$, then $(A,B)$ satisfies (QP) with subgroup $K=K_2\leqslant G$.

\section{Consequences of Theorem \ref{main}}\label{corollarysection}

\subsection{Conclusion (QP) and critical pairs for finite groups}\label{remark2}  Sur-critical pairs (\S \ref{critdef}) satisfying (QP) of Theorem \ref{main} can be further described in terms of certain critical pairs for finite groups, as the following corollary shows.

\begin{corollary}\label{QPfour}  With the hypotheses and notation of Theorem \textup{\ref{main}}, if $(A,B)$ satisfies conclusion \textup{(QP)}, then with $K$ being the subgroup of that conclusion,
\begin{enumerate}
\item[(QP.4)]  $m(A+B+K)=m(A+K)+m(B+K)-m(K)$.
\end{enumerate}
\end{corollary}
Viewing $A'=A+K$ and $B'=B+K$ as subsets of the discrete group $G/K$, (QP.4) means $|A'+B'|=|A'|+|B'|-1$, while (QP.2) says that $A'+B'$ has a unique expression element.  The results of \cite{Kemperman60} may be used to classify such pairs $(A',B')$; see \S 2 of \cite{GrynkiewiczKemperman} for exposition.

\begin{proof} Let $A=A_1\cup A_0, $ $B=B_1\cup B_0$ constitute a pair satisfying (QP) and $m_*(A+B)=m(A)+m(B)$, and let $K$ be the corresponding subgroup. Note that \begin{align}\label{aka1k}
m(A+K)=m(A_1)+m(K) = m(A)-m(A_0)+m(K),
\end{align} and a similar identity holds for $m(B+K)$.   Since $A_0+B_0+K$ is a unique expression element of $A+B+K$ in $G/K$, we have
\begin{align*}
m(A+B+K)&= m_*(A+B)-m_*(A_0+B_0)+m(K)\\
&=m(A)-m(A_0)+m(B)-m(B_0)+m(K) \\
&=[m(A+K)-m(K)]+[m(B+K)-m(K)]+m(K)  \\
&= m(A+K)+m(B+K)-m(K),
\end{align*}
where the where the second line uses the hypothesis and (QP.3) of Theorem \ref{main}, and the third uses (\ref{aka1k}).  \end{proof}

As a counterpart to Corollary \ref{QPfour}, the following procedure produces sur-critical pairs satisfying (QP) from critical pairs for finite groups.

\smallskip

Let $K\leqslant G$ be a compact open subgroup so that the quotient $F=G/K$ is discrete; write $\phi:G\to F$ for the quotient map.   Let $(A',B')$ be a critical pair for $F$ such that $A'+B'$ has a unique expression element and $|A'+B'|=|A'|+|B'|-1$.  Pick $a'\in A'$, $b'\in B'$ so that $a'+b'$ is a unique expression element of $A'+B'$.  Define $A_1':=A'\setminus \{a'\}$, and $B_1':=B'\setminus \{b'\}$.  Set $A_1=\phi^{-1}(A_1')$, $B_1=\phi^{-1}(B_1')$, and choose $C\subseteq K$ and $D\subseteq K$ such that $m_*(C+D)=m(C)+m(D)$.  Let $A_0=\phi^{-1}(a')+C$, $B_0=\phi^{-1}(b')+D$, and let $A=A_1\cup A_0$, $B=B_1\cup B_0$.   Then
\begin{align*}
m(A)=m(K)(|A'|-1)+m(C),\  m(B)=m(K)(|B'|-1)+m(D),
  \end{align*}
and
\begin{align*}
m_*(A+B)&=m(K)(|A'+B'|-1)+m_*(C+D)\\
&=m(K)(|A'|+|B'|-2)+m(C)+m(D)\\
&=m(A)+m(B).
\end{align*}

\subsection{Measurability and topology of $A$, $B$, and $A+B$.}\label{essreg}
\subsubsection{Essential regularity.}  For $S\subseteq G$, let $\overline{S}$ denote the topological closure of the set $S$, and let $\operatorname{int}S$ denote the interior of $S$.  Call a set $S\subseteq G$ \textit{essentially regular} if $m_*(S)=m(\overline{S})=m(\operatorname{int} \overline{S})$, and call a pair $(A,B)$ \textit{essentially regular} if $A$, $B$, and $A+B$ are essentially regular.  Proposition \ref{Kconnected} implies $(A,B)$ is essentially regular when $G$ is a compact connected group and $A$, $B\subseteq G$ have $m(A)>0$, $m(B)>0$, and $m(A+B)=m(A)+m(B)$.  Corollary \ref{measurableremark} below shows that restricting the measures of $A$ and $B$ in Theorem \ref{main} can guarantee that $(A,B)$ is essentially regular.   For this we need the following lemma.

\begin{lemma}\label{interior}  If for all $\varepsilon>0$, $(A,B)$ has a quasi-periodic decomposition \textup{(\S \ref{qpdef})} with respect to a compact open subgroup $K$ having $0<m(K)<\varepsilon$, then $(A,B)$ is essentially regular.
\end{lemma}

\begin{proof}  Let $\varepsilon>0$.  For a given quasi-period $K$ of $(A,B)$ having $0<m(K)<\varepsilon$, write $A=A_1\cup A_0$, where $A_1\sim A_1+K$ and $A_0$ is contained in a coset of $K$.  Then $m(\overline{A})\leq m(A+K)\leq m(A)+\varepsilon$, while $m(\operatorname{int}\overline{A})\geq m(A+K)-m(K)\geq m(A)-\varepsilon$.
Letting $\varepsilon\to 0$, we get that $m(A)=m(\overline{A})=m(\operatorname{int}\overline{A})$. The same argument applied to $B$ and $A+B$ shows that $B$ and $A+B$ are essentially regular.  \end{proof}

\subsubsection{Refinement of Theorem \textup{\ref{main}}} In conclusion (QP) of Theorem \ref{main}, possibly $m(A_0)=0$, $m(B_0)=0$, or both.  When $m(A_0)>0$ and $m(B_0)>0$, (QP.3) guarantees that Theorem \ref{main} applies to the pair $(A_0,B_0)$.  This observation yields the following corollary.

\begin{corollary}\label{roman}  With the hypotheses and notation of Theorem \textup{\ref{main}}, at least one of the following is true\textup{:}
\begin{enumerate}
\item[\textup{(I)}]  One of \textup{(P)}, \textup{(E)}, or \textup{(K)} holds.
\item[\textup{(II)}]  Conclusion \textup{(QP)} holds, and $(A_0,B_0)$ satisfies \textup{(K)}.
\item[\textup{(III)}]  For all $\varepsilon>0$, there is a compact open subgroup $H\leqslant G$  having $m(H)<\varepsilon$ and $(A,B)$ satisfies \textup{(QP)} with $K=H$.
\item[\textup{(IV)}]  $(A,B)$ satisfies \textup{(QP)} with one at least one of $m(A_0)=0$ or $m(B_0)=0$.
\end{enumerate}
\end{corollary}

\begin{proof}  If $(A,B)$ satisfies (P), (E), or (K), we have (I). If not, inductively form a sequence of pairs $(A^{(n)}, B^{(n)})$  and subgroups $K^{(n)}\leqslant G$ as follows:  let $(A^{(0)}, B^{(0)})=(A,B)$, and let $K^{(0)}=G$.  Suppose $(A^{(j)},B^{(j)})$ and $K^{(j)}$ are defined for $j=0,\dots,n$ and for each $j=1,\dots, n$,
\begin{enumerate}
\item[(1)]
$A=A'_1\cup A^{(j)}$, $B=B'_1\cup B^{(j)}$ is a quasi-periodic decomposition of $(A,B)$ with respect to $K^{(j)}$ satisfying (QP) of Theorem \ref{main}, so that
\item[(2)] $A^{(j)}$ and $B^{(j)}$ are each contained in cosets of $K^{(j)}$,
\item[(3)] $m_*(A^{(j)}+B^{(j)})=m(A^{(j)})+m(B^{(j)})$, and
\item[(4)] $K^{(j)}$ is a proper subgroup of $K^{(j-1)}$.
\end{enumerate}
We will construct $(A^{(n+1)},B^{(n+1)})$ and $K^{(n+1)}\leqslant K^{(n)}$ satisfying (1)-(4), or show that one of (I), (II), or (IV) holds.

\smallskip

If one or both of $m(A^{(n)})=0$ or $m(B^{(n)})=0$, we have conclusion (IV).  Otherwise, apply Theorem \ref{main} to $(A^{(n)},B^{(n)})$.  If $(A^{(n)},B^{(n)})$ satisfies (P) or (E), then so does $(A,B)$, and we have (I).  If ($A^{(n)}, B^{(n)})$ satisfies (K), we have (II).  Otherwise, $(A^{(n)},B^{(n)})$ satisfies (QP), so take $K^{(n+1)}$ to be corresponding subgroup $K$, and let $A^{(n)}= A^{(n)}_1\cup A^{(n)}_0$, $B^{(n)}=B^{(n)}_1\cup B^{(n)}_0$ be the corresponding decompositions. Observe that $K^{(n+1)}$ must be a proper subgroup of $K^{(n)}$, so that $m(K^{(n+1)})\leq \frac{1}{2}m(K^{(n)})$.  Now take $A^{(n+1)}=A^{(n)}_0$ and $B^{(n+1)}=B^{(n)}_0$, so that $(A^{(n+1)},B^{(n+1)})$ satisfies (1)-(4) with $j=n+1$ (cf.~\S \ref{repeateddef}).

\smallskip

If the above construction terminates, we conclude (I), (II), or (IV).  Otherwise, we have for each $n$ quasi-periodic decompositions with respect to $K^{(n)}$, and we conclude (III).
\end{proof}

\begin{corollary}\label{measurableremark}  Suppose $(A,B)$ satisfies the hypotheses of Theorem \textup{\ref{main}}.
\begin{enumerate}
\item[\textup{(a)}]  If $A+B$ is not measurable, then $(A,B)$ satisfies \textup{(QP)} with at least one of $m(A_0)=0$ or $m(B_0)=0$.

\smallskip

\item[\textup{(b)}]  If $(A,B)$ is not essentially regular, then $(A,B)$ satisfies conclusion \textup{(E)} of Theorem \textup{\ref{main}}, or $(A,B)$ satisfies conclusion \textup{(QP)} of Theorem \textup{\ref{main}} with one or both of $m(A_0)=0$ or $m(B_0)=0$.

\smallskip

\item[\textup{(c)}]  If $m(A)$, $m(B)$, and $m(A+B)$ are all irrational numbers, then $(A,B)$ is essentially regular.
\end{enumerate}
\end{corollary}

\begin{proof}    (a)  In conclusions (I) and (II) of Corollary \ref{roman}, it is routine to check that $A+B$ is measurable.  If (III) holds in Corollary \ref{roman}, Lemma \ref{interior} implies $A+B$ is measurable.

\smallskip

\noindent  (b)  If (P) or (K) holds, then $(A,B)$ is essentially regular.  If (II) or (III) of Corollary \ref{roman} holds, then $(A,B)$ is essentially regular.  The only remaining alternatives are (E) of Theorem \ref{main} or (IV) of Corollary \ref{roman}.

\smallskip

\noindent (c)  If $m(A+B)$ is irrational, then neither conclusion (P) nor (E) can hold in Theorem \ref{main}.  If $m(A)$ and $m(B)$ are irrational, then (QP) cannot hold with one of $m(A_0)=0$ or $m(B_0)=0$.  Now $(A,B)$ is essentially regular, by Part (b) of the present corollary. \end{proof}

\subsection{Further description of extendible pairs}\label{Eremark}  If $(A,B)$ is a sur-critical pair satisfying conclusion (E) of Theorem \ref{main}, Proposition \ref{Ktopinequality} implies the group $H:=H(A+B)$ is compact and open and $A+B\sim A+B+H$.  Furthermore, exactly one of the following holds:
\begin{enumerate}
\item[(E.1)]  $A+B=A+B+H$.

\smallskip

\item[(E.2)]  $(A,B)$ is complementary with respect to $H$ and $A+B\neq A+B+H$, or $(A,B)$ satisfies conclusion (QP) of Theorem \ref{main}, with $K=H$, $m(A_0)+m(B_0)=m(H)$, and $A_0+B_0\neq A_0+B_0+H.$
\end{enumerate}
When $m(A_0)>0$ and $m(B_0)>0$ in (E.2), the pair $(A_0,B_0)$ is complementary with respect to the subgroup $H$; see \S \ref{complementarydef} for further description.

To obtain this classification, fix $A'\supseteq A$ and $B'\supseteq B$ such that $m(A')+m(B')>m(A)+m(B)$ and $m(A'+B')=m(A+B)$, and write $H$ for $H(A'+B')$.  By Proposition \ref{Ktopinequality},
\begin{align*}
m(A'+B')=m(A'+H)+m(B'+H)-m(H),
\end{align*}
so  $m(A)+m(B)=m(A+B)=m(A'+H)+m(B'+H)-m(H)$.  Rearranging, we get
\begin{align}\label{holes}
m(A'+H)-m(A)+m(B'+H)-m(B)=m(H).
\end{align}
Let $A=\bigcup_{i=1}^n A_i$ and $B=\bigcup_{j=1}^m B_j$ be $H$-coset decompositions of $A$ and $B$.  Since $A\subseteq A'+H$ and $B\subseteq B'+H$, (\ref{holes}) implies
\begin{align}\label{surholes}
m(A_i)+m(B_j)\geq m(H)
\end{align}
 for each $i$ and $j$ (otherwise the left-hand side of (\ref{holes}) would be larger than $m(H)$).  If the inequality (\ref{surholes}) is strict for each pair $i$ and $j$, we have (E.1).  Otherwise, take $i$ and $j$ so that equality holds in (\ref{holes})  and set $A_0=A_i$, $B_0=B_j$.  If $A=A_0$ and $B=B_0$, then $(A,B)$ is complementary with respect to $H$.  If not, we set $A_1=A\setminus A_0$ and $B_1=B\setminus B_0$, we verify (QP.1)-(QP.3) in Theorem \ref{main}.   Equation (\ref{holes}) implies $A_1\sim A_1+H$ and $B_1\sim B_1+H$, so (QP.1) holds. If $A+B\neq A+B+H$, we conclude that $A_0+B_0+H$ is a unique expression element of $A+B+H$ in $G/H$ and we have (QP.2).  (QP.3) now follows from $A+B\sim A+B+H$.

\section{Examples}\label{examples}  We list some examples of sur-critical pairs to show that each alternative in Theorem \ref{main} can occur, and that none of the alternatives can be omitted.  We do not attempt to exhaustively construct all possible sur-critical pairs.

\subsection{Periodic pairs}\label{periodicexamples}  If $F$ is a finite group, every sur-critical pair for $F$ satisfies (P) and not (K), and every nonextendible sur-critical pair satisfies (P) but not (E).  A specific example with $F=\mathbb Z/11\mathbb Z$ is $A=\{0,1\}$, $B=\{0,3\}$, so that $A+B=\{0,1,3,4\}$, and $(A,B)$ satisfies (P) but not (E) or (K).  Every periodic sur-critical pair $(A,B)$ for a compact group $G$ has the form $A=\phi^{-1}(C)\setminus N$, $B=\phi^{-1}(D)\setminus N$, where $(C,D)$ is a sur-critical pair for a finite group $F$, $\phi:G\to F$ is a continuous surjective homomorphism,  and $m(N)=0$, so the periodic sur-critical pairs for an arbitrary compact group $G$ are classified in \cite{Grynkiewicz}.

\subsection{Extendible pairs} The following example satisfies (E), but not (P), (K), or (QP).  Let $G=(\mathbb Z/ 17\mathbb Z)\times \mathbb T$ and let $A=\{1,3,5,7\}\times [0,0.8]$, $B=\{0,2\}\times [0,0.9]$, so that $A+B=\{1,3,5,7,9\}\times \mathbb T$.  Then $m(A)=3.2/17$, $m(B)=1.8/17$, and $m(A+B)=5/17$.  Also $A+B=A'+B'$, where $A'=\{1,3,5,7\}\times \mathbb T$ and $B'=\{0,2\} \times \mathbb T$.

\smallskip

To find examples satisfying (E.1)  (\S \ref{Eremark}),  fix a pair of sets  $A'$, $B'\subseteq G$ satisfying $m(A'+B')<m(A')+m(B')$, and let $H=H(A'+B')$.  Choose any pair of subsets $A\subseteq A'$, $B\subseteq B'$ with $m(A)+m(B)=m(A'+B')$, subject to the condition that $m((a+H)\cap A)+m((b+H)\cap B)>m(H)$ for all $a\in A'$ and $b\in B'$.  Then $A+B=A'+B'$, so $m(A+B)=m(A)+m(B)$.

\smallskip

To find examples satisfying (E.2) but not (E.1),  recall the construction in \S \ref{remark2}.  Let $(A,B)$ satisfy conclusion (QP) of Theorem \ref{main} with $m(A_0)+m(B_0)=m(K)$, but $A_0+B_0\neq A_0+B_0+K$.   To form a specific example of this construction with $G=(\mathbb Z/ 15\mathbb Z)\times \mathbb T$, let $A=(\{1,3,5\}\times \mathbb T)\cup (\{7\}\times S)$, where $S\subseteq \mathbb T$ is any measurable set, and let $B=(\{0,2\}\times \mathbb T)\cup (\{4\}\times (-S^c))$.  This example satisfies (E) and (QP) but not (P) or (K).

\subsection{Pairs arising from $\mathbb T$} Let $G$ be a compact group which is not totally disconnected. Let $(A,B)$ have the form $A=\chi^{-1}(I)$,  $B=\chi^{-1}(J)$, where $H\leqslant G$ is a compact open subgroup, $\chi:H\to \mathbb T$ is a continuous surjective homomorphism, $I,J\subseteq \mathbb T$ are intervals, and $m_{\mathbb T}(I)+m_{\mathbb T}(J)<1$.  One can verify that $(A,B)$ is a sur-critical pair satisfying (K) but not (P), (E), or (QP).

\subsection{A quasi-periodic pair} This example will satisfy (QP) but not (P), (E), or (K).

\smallskip

Let $G=\mathbb Z_7$, the $7$-adic integers with the usual topology, and consider $\mathbb Z$ as a subset of $\mathbb Z_7$ in the usual way.  Define the set $C\subseteq \mathbb Z$ by
\begin{align*}
C:=(\{0,1\}+7\mathbb Z)\cup\Bigl(2+7((\{0,1\}+7\mathbb Z)\cup (2+7(\cdots)))\Bigr),
\end{align*}
so that $C=(\{0,1\}+7\mathbb Z)\cup (2+7C)$.  Let $A$ be the closure of $C$ in $\mathbb Z_7$, and let $B=A$.  Then $
m(A)=m(B)=\sum_{n=1}^\infty \frac{2}{7^n}=1/3$. Note that $A$ has a quasi-periodic decomposition $A_1\cup A_0$ where $A_1=\overline{\{0,1\}+7\mathbb Z}$, $A_0=A\setminus A_1$.  The sumset $A+B$ is the closure of $C+C$ in $\mathbb Z_7$.  Note that
 \begin{align}\label{C+C}
C+C=(\{0,1,2,3\}+7\mathbb Z)\cup\Bigl(4+7((\{0,1,2,3\}+7\mathbb Z)\cup (4+7(\cdots)))\Bigr),
 \end{align}
so $m(A+B)=2/3=m(A)+m(B)$.  From (\ref{C+C}) one can check that $H(A+B)=\{0\}$, so (P) and (E) fail.  Since $G$ is totally disconnected, (K) cannot hold.

\subsection{More quasi-periodic pairs.} Let $G'$ be an infinite compact abelian group with Haar measure $m'$, and let $G=(\mathbb Z/4\mathbb Z)\times G'$.  Fix sets $A_0'$, $B_0'\subseteq G'$ satisfying $m_{*}'(A_0'+B_0')=m'(A_0')+m_0(B_0')$, let $A_1=\{0\}\times G'$, $B_1=\{0\}\times G'$, and let $A_0=\{1\}\times A_0'$, $B_0=\{1\}\times B_0'$.  Let $A=A_1\cup A_0$, $B=B_1\cup B_0$.  Then $m_*(A+B)=m(A)+m(B)$, but depending on our choice of $A_0$ and $B_0$, $A+B$ may not be measurable.  If $A_0'$ is a singleton, then $B_0'$ may be an arbitrary measurable subset of $G'$.

\section{Lemmas}\label{lemmas}

We fix a compact abelian group $G$ with Haar measure $m$.  Unless stated otherwise, sets $A$ and $B$ are assumed to be subsets of $G$.

\subsection{Consequences of Proposition \ref{Ktopinequality}}

We need Lemma \ref{overspill} and Corollary \ref{consequence} in many subsequent proofs.  See \S \ref{Hdef} for notation.

\begin{lemma}\label{overspill}
If $m_*(A+B)<m(A)+m(B)$, then for all $a\in A$ and all $b\in B$,
\begin{align}
\label{os1} m((a+H)\cap A)+m_*(A+B)&\geq m(A)+m(B),\\
\label{os2} m((b+H)\cap B) +m_*(A+B)&\geq m(A)+m(B),
\end{align}
where $H=H(A+B)$.
\end{lemma}

\begin{proof}  We prove (\ref{os1}).  Write $m(A)=\sum_{i=1}^n m(A_i)$, where $A=\bigcup_{i=1}^n A_i$ is an $H$-coset decomposition of $A$.  Fix $a\in A$, and let $j$ be such that $(a+H)\cap A=A_j$. Then $m(A_j)=m(A)-\sum_{i\neq j} m(A_i)$, while $m(A+H)=\sum_{i=1}^n m(H)$.   Now Proposition \ref{Ktopinequality} implies
\begin{align*}
m(A_j)+m_*(A+&B)=\Bigl(m(A)-\sum_{i\neq j}m(A_i)\Bigr)+m(A+H)+m(B+H)-m(H)\\
&= \Bigl(m(A)-\sum_{i\neq j}m(A_i)\Bigr) + \Bigl(\sum_{i} m(H)\Bigr)+m(B+H)-m(H)\\
&= m(A)+\Bigl(\sum_{i\neq j} m(H)-m(A_i)\Bigr)+m(H)+m(B+H)-m(H)\\
&\geq m(A)+m(B+H)\\
&\geq m(A)+m(B).
\end{align*}
Equation (\ref{os2}) follows by symmetry. \end{proof}

We write $C\subset_m D$ below to mean $m(C\setminus D)=0$.
\begin{corollary}\label{consequence}  If $m_*(A+B)< m(A)+m(B)$, let $H:=H(A+B)$.  Then
\begin{align}\label{mover1}
\{z\in G:z+B\subset_m A+B\}=\{z\in G: z+B\subseteq A+B\}=A+H,
\end{align}
\begin{align}\label{mover2}
\{z\in G: A+z\subset_m A+B\}=\{z\in G: A+z \subseteq A+B\}=B+H.
\end{align}
\end{corollary}

\begin{proof}  For the first equality in (\ref{mover1}), note that $z+B\subset_m A+B$ implies $z+B\subseteq A+B$, by Lemma \ref{overspill} and the fact that $A+B$ is a union of cosets of $H$.  For the second equality in (\ref{mover1}), we argue by contradiction.  Fix $z\in G$ such that $z+B\subseteq A+B$, but $z\notin A+H$.   Let $A'=A\cup (z+H)$.     Then $A'+B=A+B$, so that $H(A'+B)=H(A+B)$, while $m(A'+H)=m(A+H)+m(H)$.  Applying Proposition \ref{Ktopinequality} to $(A',B)$ yields
\begin{align*}
m(A'+B)&=m(A'+H)+m(B+H)-m(H)\\
&= m(A+H)+m(B+H),
\end{align*}
which contradicts the assumption that $m_*(A+B)<m(A)+m(B)$.  Equation (\ref{mover2}) follows by symmetry.  \end{proof}

\begin{lemma}\label{Hholes}  If $m_*(A+B)<m(A)+m(B)$ and $H:=H(A+B)$, then for all $a\in A$ and $b\in B$,
\begin{align}\label{organize}
m((a+H)\cap A)+m((b+H)\cap B)>m(H).
\end{align}
\end{lemma}

\begin{proof}  If inequality (\ref{organize}) fails, then $m(A+H)+m(B+H)\geq m(A)+m(B)+m(H)$, so inequality (\ref{critical}) in Proposition \ref{Ktopinequality} implies $m_*(A+B)\geq m(A)+m(B)$, contradicting the hypothesis.
\end{proof}

In the next corollary we use the elementary fact that if $G$ is compact, and $C$, $D\subseteq G$ have $m(C)+m(D)>1$, then $C+D=G$.  This fact follows from the observation that $t\in C+D$ if and only if $(t-C)\cap D\neq \varnothing$, while $t-C$ and $D$ cannot be disjoint if $m(C)+m(D)>m(G)$.  Consequently, if $H\leqslant G$ is a compact open subgroup and $C,$ $D\subset G$ are each contained in a coset of $H$ while  $m(C)+m(D)>m(H)$, then $C+D=C+D+H$.

\begin{corollary}\label{differencetosum}  If $m_*(A+B)<m(A)+m(B)$, then $A-B$ is periodic with period $H:=H(A+B)$.
\end{corollary}

\begin{proof} Let $A=\bigcup_{i=1}^n A_i$ and $B=\bigcup_{i=1}^m B_i$ be $H$-coset decompositions of $A$ and $B$, so that $A-B=\bigcup_{i,j} A_i-B_j$.  Inequality (\ref{organize}) implies $m(A_i)+m(B_j)>m(H)$ for each $i$ and $j$, so $A_i-B_j= A_i-B_j+H$ for each $i$ and $j$, implying $A-B=A-B+H$.
\end{proof}

\subsection{Reducible pairs}  We now dispense with the case of Theorem \ref{main} where $(A,B)$ is reducible (\S \ref{reducibledef}). For the next two lemmas and the following corollary, we assume that $m(A)>0$, $m(B)>0$, and $(A,B)$ satisfies $m_*(A+B)=m(A)+m(B)$.

\begin{lemma}\label{reduciblelemma1} If $A'\subseteq A$ is such that $m(A')=m(A)$ and $m_*(A'+B)<m_*(A+B)$, then $A'+H(A'+B)\subset_m A$ and at least one of the following holds\textup{:}

\begin{enumerate}
\item[(i)]  $(A,B)$ satisfies \textup{(QP)} of Theorem \textup{\ref{main}}, with  $K=H:=H(A'+B)$,  \\ $A_1=A\cap(A'+H)$, and $A_0=A\setminus A_1$, so that $m(A_1)=m(A)$ and $m(A_0)=0$.

\smallskip

\item[(ii)]  $B\sim B+H(A'+B)$.
\end{enumerate}
If $(A,B)$ is nonextendible, then {\textup{(i)}} holds.
\end{lemma}

\begin{proof}  If $A$, $B$, $A'$ and $H$ are as in the hypothesis, Proposition \ref{Ktopinequality} implies $H$ is compact and open and $A'+B$ is a union of cosets of $H$. Corollary \ref{consequence} implies that whenever $a_0\in A$ is such that $a_0+B \not\subseteq A'+B$, there is a $b_0\in B$ such that
\begin{align}\label{Bover2}
m((a_0+B_0)\setminus (A'+B))>0,
\end{align}
where $B_0:=(b_0+H)\cap B$.   Fix $a_0\in A$, $b_0\in B$, and $B_0$ satisfying (\ref{Bover2}), and let $B_1=B\setminus B_0$.  Lemma \ref{overspill} and the hypothesis $m(A')=m(A)$ imply
\begin{align}\label{Bover1}
m(B_0)+m(A'+B)\geq m(A')+m(B) =m_*(A+B).
\end{align}

\begin{claim1}  $A'+H \subset_m A$ and $B_1+H\subset_m B$.
\end{claim1}

\noindent\textit{Proof of Claim \textup{1}.}  By the definition of $B_0$ and the fact that $A'+B=A'+B+H$,  $A'+B$ is disjoint from $a_0+B_0$.    Using (\ref{Bover1}) we write
\begin{align}\label{innermeasure}
m_*(A+B)=m(A'+B)+m(B_0).
\end{align}
Then
\begin{align}\label{arranging}\begin{split}m(A')+m(B)&=m_*(A+B)\\
&=m(A'+B)+m(B_0)\\
&=m(A'+H)+m(B+H)-m(H)+m(B_0)\\
&=m(A'+H)+m(B_1+H) +m(B_0),\end{split}
\end{align}
where the second line is (\ref{innermeasure}), the third line is from Proposition \ref{Ktopinequality}, and the fourth line is due to $B=B_1\cup B_0$ being a disjoint union and $B_0$ being contained in a coset of $H$. Subtracting $m(A')+m(B)$ from the first and last lines of (\ref{arranging}), we find
\begin{align}\label{arranged}
[m(A'+H)-m(A')] + [m(B_1+H)+m(B_0)-m(B)]=0
\end{align}
Each bracketed summand in (\ref{arranged}) is nonnegative, so
\begin{align}\label{A'}
m(A')=m(A'+H)
\end{align}
and
\begin{align}\label{B'}
m(B)=m(B_1+H)+m(B_0).
\end{align}
Since $(B_1+H) \cap B_0=\varnothing$ and $B=B_1\cup B_0$, (\ref{A'}) and (\ref{B'}) prove the claim.  \hfill $\square$

\smallskip

We proceed based on whether $m(B_0)<m(H)$.

\subsubsection*{Case \textup{1}. $m(B_0)=m(H)$} In this case Claim 1 implies $B\sim B+H$, and we conclude (ii) in the statement of the lemma.

\subsubsection*{Case \textup{2}.  $m(B_0)<m(H)$}  Let $A_0:=A\setminus (A'+H)$.  We will show conclusion (i) of the lemma holds.

\begin{claim2}  $A_0$ is contained in a coset of $H$ and $A_0+B_0+H$ is a unique expression element of $A+B+H$ in $G/H$.
\end{claim2}

\noindent\textit{Proof of Claim \textup{2}.}  Observe that $A+B_1\subseteq A'+B$, since otherwise Claim 1 implies $m_*(A+B)\geq m(A'+B)+m(H) >m(A'+B)+m(B_0)$, contradicting (\ref{innermeasure}).  Furthermore, $A'+H$ is the set $\{z\in G: z+B\subset_m A'+B\}$, by Corollary \ref{consequence}.  Thus, $A_0$ is the set of $a\in A$ such that $a+B_0\not\subseteq A'+B$.  Now if $a$, $a'\in A_0$,  then $a+B_0\sim a'+B_0$, by (\ref{innermeasure}).  Since $B_0$ is contained in a coset of $H$ and $m(B_0)>0$, the similarity $a+B_0\sim a'+B_0$ implies $a-a'\in H$.   We conclude that $A_0$ is contained in a coset of $H$.  Now the containment $A+B_1\subseteq A'+B$ implies $A_0+B_0+H$ is a unique expression element of $A+B+H$ in $G/H$, as claimed.  \hfill $\square$

\smallskip

Writing $A_1=A\cap(A'+H)$, and maintaining $A_0$, $B_1$, and $B_0$ as defined above,  the partition $A=A_1\cup A_0$ is a quasi-periodic decomposition with quasi-period $H$, while $B_1\sim B_1+H$, and $B_0$ is contained in a coset of $H$ (possibly $B_1=\varnothing$), so we have verified that (QP.1) and (QP.2) hold.  To verify (QP.3), consider the partition $A+B= (A_1+B)\cup (A_0+B_0)$,  so that  (\ref{innermeasure}) implies $m_*(A_0+B_0)=m(B_0)$.  Since $m(A_0)=0$, we then have $m_*(A_0+B_0)=m(A_0)+m(B_0)$. This concludes the analysis of Case 2.

\smallskip

 Now suppose $(A,B)$ is nonextendible.  If Case 2 holds, we have conclusion (i) as desired.   If Case 1 holds we derive a contradiction:  from $B+H\sim B$ and the nonextendibility of $(A,B)$ we find $A\sim A+H$.  Since $H$ is compact and open,  the similarity $A\sim A+H$ implies  $m(A''+B)=m_*(A+B)$ whenever $A''\subseteq A$ has $m(A'')=m(A)$, contradicting the hypothesis of the lemma.    \end{proof}

\begin{lemma}\label{reduciblelemma2}
 If $(A,B)$ is reducible and nonextendible, then $(A,B)$ has a quasi-periodic decomposition $A=A_1\cup A_0$, $B=B_1\cup B_0$ satisfying conclusion \textup{(QP)} of Theorem \textup{\ref{main}}, and at least one of $m(A_0)=0$ or $m(B_0)=0$.
\end{lemma}

\begin{proof}  Let $A'\subseteq A$ and $B'\subseteq B$ be such that $m(A')+m(B')=m(A)+m(B)$ and $m_*(A'+B')<m_*(A+B)$.   Write $H:=H(A'+B')$, and note that $H$ is compact and open, by Proposition \ref{Ktopinequality}.  We consider two cases.

\subsubsection*{Case \textup{1}. $m_*(A'+B)<m_*(A+B)$ or $m_*(A+B')<m_*(A+B)$}  We apply Lemma \ref{reduciblelemma1}  and we are done.

\subsubsection*{Case \textup{2}. $m_*(A'+B)=m_*(A+B')=m_*(A+B)$}  In this case Lemma \ref{reduciblelemma1} applies to the pairs $(B,A')$ and $(A,B')$, so
\begin{align}\label{simsim}
B'+H\sim B \text{ and } A'+H \sim A.
\end{align}
We claim that $A_0:=A\setminus (A'+H)$ is contained in a coset of $H$.  Let $t$ be the number of cosets of $H$ occupied by $A\setminus (A'+H)$, so that \begin{align}\label{tee}
m(A+H)-m(A')=t\cdot m(H).
 \end{align}
 We aim to show that $t=1$, so that $A_0$ is contained in a coset of $H$.  We have $t>0$, as the hypothesis of Case 2 implies $A+H\neq A'+H$.

\smallskip

Consider the pair $(A+H,B')$.  By (\ref{simsim}), $A+H+B'\sim A+B'\sim A+B$.  Then
\begin{align}\label{computation}
\begin{aligned}m((A+H)+B') &=m_*(A+B)\\
&=m(A)+m(B)  && \text{(by hypothesis)}\\
&=m(A'+H)+m(B')  && \text{(by (\ref{simsim}))}\\
&=m(A+H)-t\cdot m(H)+m(B')  && \text{(by (\ref{tee}))}\\
&<m(A+H)+m(B'),
\end{aligned}
\end{align}
Applying Proposition \ref{Ktopinequality} to the pair $(A+H,B')$ we get that $K:=H((A+H)+B')$ is compact and open,
\begin{align}\label{+K}
A+H+B'+K=  A+H+B'
\end{align} and
\begin{align}\label{-m(K)}
m(A+H+B')=m(A+H+K)+m(B'+K)-m(K).
\end{align}
We want to show $K\leqslant H$.  The nonextendibility of $(A,B)$, (\ref{simsim}),  and (\ref{+K}) together imply $B'+K\subset_m B$.   It follows that $B'+K\subset_m E$ whenever $E\sim B$.  Consequently $B'+K\subset_m B'$, and we conclude  $A'+B'\sim A'+B'+K$, so $K\leqslant H$ by the definition of $H$.  Now (\ref{-m(K)}) becomes
\begin{align}\label{KandH}
m(A+H+B')=m(A+H)+m(B')-m(K).
\end{align}  Comparing (\ref{KandH})  with the fourth line of (\ref{computation}), we get $t\cdot m(H)=m(K)$.  This implies $t=1$, since $K\subseteq H$ and $m(K)>0$.

\smallskip

We have shown that $A_0$ is contained in a coset of $H$, so $A_1:=A\setminus A_0$ gives the desired quasi-periodic decomposition of $A$.  Reversing the roles of $A$ and $B$, we find the corresponding quasi-periodic decomposition of $B$: $B_0=B\setminus (B'+H)$, $B_1=B\setminus B_0$.  We now show that $A_1\cup A_0$ and $B_1\cup B_0$ satisfy (QP.2) and (QP.3).

\smallskip

Observe that $m(A_0)=m(B_0)=0$, by our choice of $A_0$ and $B_0$.  Nonextendibility of $(A,B)$ implies $A_0+B_0+H$ is a unique expression element of $A+B+H$ in $G/H$.  To see that $m_*(A_0+B_0)=0$, write $A+B$ as $(A_1+B)\cup (A_0+B_0)$, so that $m_*(A+B)=m(A_1+B)+m_*(A_0+B_0)$.  By the hypothesis of Case 2, we have $m_*(A_1+B)=m_*(A+B)$, so $m_*(A_0+B_0)=0=m(A_0)+m(B_0)$.  \end{proof}

\begin{corollary}\label{reduciblecor}  If $(A,B)$ is reducible and $m(H(A+B))=0$, then $(A,B)$ satisfies conclusion \textup{(QP)} of Theorem \textup{\ref{main}} with  at least one of $m(A_0)=0$ or $m(B_0)=0$.
\end{corollary}

\begin{proof}  If  $m(H(A+B))=0$, then $(A,B)$ is nonextendible, by Proposition \ref{Ktopinequality}.  The conclusion now follows from Lemma \ref{reduciblelemma2}. \end{proof}

\subsection{Auxiliary lemmas}  The next two lemmas handle technicalities arising repeatedly in the next subsection.  We continue to assume $(A,B)$ satisfies the hypotheses of Theorem \ref{main}.

\begin{lemma}\label{easy}  Suppose $(A,B)$ is irreducible and nonextendible.  Assume further that $H:=H(\tilde{A}+\tilde{B})$ has positive Haar measure for some $\tilde{A}\subseteq A$, $\tilde{B}\subseteq B$  such that $\tilde{A}+\tilde{B}$ is measurable, $m(\tilde{A})=m(A)$ and $m(\tilde{B})=m(B)$. Then $A$ and $B$ are periodic with period $H$, or $(A,B)$ satisfies \textup{(QP)} of Theorem \textup{\ref{main}} with $K=H$.
\end{lemma}

\begin{proof}  Let $\delta=m(H)$.  Since $m(H)>0$, we have $m((\tilde{A}+\tilde{B})\cap H_i)\in \{0,\delta\}$ for every coset $H_i$ of $H$ (cf.~\S \ref{periodicdef}).  Irreducibility of $(A,B)$ then implies
\begin{align}\label{0delta}
m_*((A+B)\cap H_i)\in\{0,\delta\} \text{ for every coset } H_i \text{ of } H.
\end{align} Combining (\ref{0delta}) with the nonextendibility of $(A,B)$, we have $m(A\cap H_i)\in \{0,\delta\}$ and  $m(B\cap H_i)\in \{0,\delta\}$ for every coset $H_i$ of $H$.  Now by the irreducibility of $(A,B)$,  there are sets $A'\subseteq A$, $B'\subseteq B$ such that
\begin{align}\label{AsimsBsims} A\sim A'+H\sim A' \text{ and } B\sim B'+H\sim B',
 \end{align}
 $A'+B'$ is measurable, and $m_*(A+B)=m(A'+B')$.  Setting  $A_0:=A\setminus (A'+H)$, we claim $A_0$ is contained in a coset of $H$.  To prove this claim, assume $A_0$ is nonempty, let $a\in A_0$, and consider the pair $(A'\cup (a+H), B')$.  Write $C$ for $A'\cup (a+H)$, so $m(C)=m(A')+m(H)$.  By (\ref{AsimsBsims}) we have $a+B'+H\subset_m A+B$, so $C+B'\subset_m A+B$.   The irreducibility of $(A,B)$ then implies
\begin{align}\label{C+B'}
C+B' \sim \tilde{A}+\tilde{B},
\end{align} so that $H(C+B')=H(\tilde{A}+\tilde{B})$.   We also have
$$
m(C+B')=m_*(A+B)=m(A)+m(B)<m(C)+m(B'),
$$
so we may apply Corollary \ref{consequence} to the pair $(C,B')$ and find
\begin{align}\label{CplusH}
C+H=\{z\in G: z+B'\subset_m C+B'\}.
\end{align}
By (\ref{C+B'}), the right-hand side of (\ref{CplusH}) is also $\{z\in G: z+B'\subset_m \tilde{A}+\tilde{B}\}$, which does not depend on $a\in A_0$.  Hence, $(A'+H)\cup (a+H)$ does not depend on the choice of $a\in A_0$.  This means that $A_0$ is contained in a coset of $H$, as claimed.

\smallskip

If $A_0=\varnothing$, then $B\sim B+H$ by the irreducibility and nonextendibility of $(A,B)$,  so $A$ and $B$ are periodic with period $H$, and we are done.   If $A_0$ is nonempty, we will show that $(A,B)$ satisfies (QP).  Set $A_1=A\setminus A_0$ to get the desired quasi-periodic decomposition $A_1\cup A_0$ of $A$ with $m(A_0)=0$.  Reversing the roles of $A$ and $B$ in the previous paragraph, we find that either $B$ is $H$-periodic or has a quasi-periodic decomposition $B_1\cup B_0$ with respect to $H$, where $m(B_0)=0$.  The nonextendibility of $(A,B)$ implies $B_0$ is nonempty, and also implies $A_0+B_0+H$ is a unique expression element of $A+B+H$ in $G/H$.    Now (\ref{0delta}) and the nonextendibility of $(A,B)$ imply $m_*(A_0+B_0)=0=m(A_0)+m(B_0)$.  \end{proof}

\begin{lemma}\label{QPtoQP}  Let $(A',B')$ be an irreducible nonextendible sur-critical pair satisfying \textup{(QP)} of Theorem \textup{\ref{main}} with decomposition $A'=A_1'\cup A_0'$, $B'=B_1'\cup B_0'$, such that $m(A_0')>0$ and $m(B_0')>0$.   If $(A,B)$ is another irreducible nonextendible sur-critical pair such that $A\sim A'$ and $B\sim B'$, then $(A,B)$ also satisfies \textup{(QP)}.
\end{lemma}

\begin{proof}   We may assume $A'\subseteq A$ and  $B'\subseteq B$, since we can replace $A'$ and $B'$ with $A'\cap A$ and $B'\cap B$ while maintaining the hypotheses of the lemma -- this follows from the irreducibility of $(A,B)$ and $(A',B')$.  Let $K$ be the quasi-period of $(A',B')$ from conclusion (QP). Assume, without loss of generality, that $A_1'\neq \varnothing$.  Define $A_0:=A\cap (A_0'+K)$ and  $A_1:= A\setminus A_0$.   We will show that $A_1\sim A_1+K$.

\smallskip

Let $a\in A_1$. We aim to show that $a+K\subseteq A_1'+K$.  Irreducibility of $(A,B)$ implies $a+B'\subset_m A'+B'$.   Since $a\notin A_0'+K$, we have $a+B'\subset_m (A_1'+B')\cup (A+B_1')$, so that $a+B'+K\subset_m A'+B'$.  Then $a+K\subset_m A'$ by the nonextendibility of $(A',B')$.   Consequently, $a+K\subseteq A_1'+K$, and we find $A_1+K=A_1'+K$.  The similarity $A\sim A'$ then implies $A_1\sim A_1'$, and $A=A_1\cup A_0$ is a quasi-periodic decomposition.  Reversing the roles of $A$ and $B$, we define $B_0=B\cap (B_0'+K)$ and $B_1=B\setminus B_0$.  Now it is routine to verify that $(A,B)$ also satisfies (QP) of Theorem \ref{main}.  \end{proof}

\subsection{Measurability and trivializing $H(A+B)$.}\label{trivializing}

\begin{lemma}\label{measurability}
If Theorem \textup{\ref{main}} holds under the additional assumption that $A+B$ is measurable,  then Theorem \textup{\ref{main}} holds in general.
\end{lemma}

\begin{proof}  We first list cases where we obtain the conclusion of Theorem \ref{main} without assuming $A+B$ is measurable.  Excluding those cases, we then show how to deduce the conclusion of Theorem \ref{main} from the special case where $A+B$ is assumed to be measurable.

\smallskip

If $(A,B)$ is extendible, then $(A,B)$ satisfies conclusion (E) of Theorem \ref{main}.  If $(A,B)$ is nonextendible and reducible,  then $(A,B)$ satisfies conclusion (QP) of Theorem \ref{main}, by Lemma \ref{reduciblelemma2}.  We may therefore assume that $(A,B)$ is nonextendible and irreducible.  Let $A'\subseteq A$ and $B'\subseteq B$ be countable unions of compact sets with $m(A')=m(A)$ and $m(B')=m(B)$, so that $A'+B'$ is measurable.

\smallskip

Irreducibility of $(A,B)$ implies $m(A'+B')=m_*(A+B)$, so $(A',B')$ satisfies the hypotheses of Theorem \ref{main}.  If $(A',B')$ satisfies conclusion (P) or conclusion (E) of Theorem \ref{main}, then $m(H(A'+B'))>0$ by Proposition \ref{Ktopinequality}.  We apply Lemma \ref{easy} and conclude that $(A,B)$ satisfies one of conclusion (P) or (QP).  If $(A',B')$ satisfies (K), so that $A'\sim a+\chi^{-1}(I)$ and $B'\sim b+\chi^{-1}(J)$, where $I$, $J\subseteq \mathbb T$ are intervals and $\chi:K\to \mathbb T$ is a continuous surjective homomorphism from a compact open subgroup $K\leqslant G$,  then a routine argument (\S \ref{intervalsdef}) shows that $A$ and $B$ must be contained in $a+\chi^{-1}(I)$, and $b+\chi^{-1}(J)$, respectively.   We conclude that $(A,B)$ satisfies (K).

\smallskip

The only remaining possibility is that $(A',B')$ satisfies (QP), but not (P) or (E).  Then $(A',B')$ is nonextendible, while irreducibility of $(A,B)$  implies that of $(A',B')$.  Writing $A'=A_1'\cup A_0'$ and $B'=B_1'\cup B_0'$ for the decomposition of $(A',B')$ in (QP) and $K$ for the corresponding subgroup, we consider different cases based on the measures of $A_0'$ and $B_0'$.

 \smallskip

 If $m(A_0')>0$ and $m(B_0')>0$,  Lemma \ref{QPtoQP} implies $(A,B)$ satisfies (QP).

 \smallskip

 If $m(A_0')=0$ and $m(B_0')>0$, or vice versa, we derive a contradiction.  Under these conditions, $(A',B')$ is reducible, since $A_0'+B_0'+K$ is a unique expression element of $A'+B'+K$ in $G/K$. Reducibility of $(A',B')$ contradicts our present assumptions.

 \smallskip

If $m(A_0')=m(B_0')=0$, the irreducibility of $(A,B)$ implies $A+B\sim A_1'+B_1'$.  Since $H(A_1'+B_1')$ has positive Haar measure, we apply Lemma \ref{easy} with $\tilde{A}=A_1'$, $\tilde{B}=B_1'$ and conclude that $(A,B)$ satisfies (P) or (QP) of Theorem \ref{main}.  \end{proof}

The next lemma is a rephrasing of Lemma 7 of \cite{Kneser56}; we include a proof for completeness.  Under the standing assumption that $G$ is compact, the group $H(S)$ below is compact, so the convention that Haar measure is normalized applies:  $m_{H(S)}(H(S))=1$.

\begin{lemma}\label{H(S)}
If $S\subseteq G$ is measurable, set $H:=H(S)$ and let
$$
E:=\{x\in G: 0<m_{H}(S-x)<1 \text{ \textup{or} } H\cap (S-x) \text{ \textup{is not} $m_{H}$\textup{-measurable}}\}.
$$ Then $m(E)=0$.
\end{lemma}

\begin{proof}  Consider the integrals
\begin{align*}
I_1&:=\int \int 1_S(x-z)1_S(x) \, dm_H(z) \, dm(x)\\
I_2&:= \int \int 1_S(x-z)1_S(x) \, dm(x) \, dm_H(z).
\end{align*}
Proposition \ref{Weil}, these integrals are well-defined, and $I_1$ can be computed as
\begin{align*}
\int \int 1_S(x-z) \, dm_H(z)\, 1_S(x) \, dm(x)=\int m_H(S-x)1_S(x)\, dm.
\end{align*}
 while $I_2 =\int m((S+z)\cap S)\, dm_H(z)=m(S)$.   By Fubini's theorem, $I_1=I_2$, meaning $\int m_H(S-x) 1_S(x)\, dm(x) =m(S)$.  Since $0\leq m_H(S-x)\leq 1$, the last equation implies
 \begin{align}\label{S-x}
 m_H(S-x)=1 \text{ for $m$-almost every $x\in S$}.
 \end{align}
 Replacing $S$ with $S^c$, (\ref{S-x}) implies $m_H(S-x)=0$ for $m$-almost every $x\in S^{c}$.  The assertion follows. \end{proof}

\begin{lemma}\label{H(A+B)}
If $(A,B)$ is irreducible and satisfies the hypotheses of Theorem \textup{\ref{main}}, $A+B$ is measurable, and $m(H(A+B))=0$, then there are measurable sets $A'\subseteq A$, $B'\subseteq B$ such that $m(A')=m(A)$, $m(B')=m(B)$, $A'\sim A'+H$, $B'\sim B'+H$, and $A'+B'+H\sim A'+B'\sim A+B$, where $H:=H(A+B)$.
\end{lemma}

\begin{proof}  Let
\begin{align*}
S_1:=\{x\in G: m_H(A-x)>0\},\ \  S_2:=\{x\in G: m_H(B-x)>0\}.
\end{align*}
By Proposition \ref{Weil}, $S_1$ and $S_2$ are measurable.  Let $A''=A\cap S_1$, $B''=B\cap S_2$;  Proposition \ref{Weil} implies $m(A'')=m(A)$ and $m(B'')=m(B)$. Let $A'\subseteq A''$ and $B'\subseteq B''$ be countable unions of compact sets having $m(A')=m(A)$ and $m(B')=m(B)$.   The irreducibility of $(A,B)$ implies  $A'+B'\sim A''+B''\sim A+B$, so that $H(A'+B')=H$. By Lemma \ref{H(S)}, the set
\begin{align*}
E=\{x\in G: 0<m_H(A+B-x)<1 \text{ or } H\cap(A+B-x) \text{ is not $m_{H}$-measurable}\}
\end{align*} has $m_G(E)=0$.  By Lemma \ref{H(S)} and the definitions of $S_A$, $S_B$, and $E$,
 \begin{align}\label{=1}
 m_H(A'+B'+H-x)\leq m_H(A''+B''-x)=1
  \end{align}
  whenever $x\notin E$ and $(A''+B''-x)\cap H$ is measurable and nonempty.  Since $A'$ and $B'$ are countable unions of compact sets and $H$ is compact, $A'+B'+H$ is measurable.  We now show $A'+B'+H\sim A'+B'$.  By Proposition \ref{Weil} and (\ref{=1}), then
 \begin{align*}
 m(A'+B'+H)&=\int m_H(A'+B'+H-x)\, dm(x) \\
 &= \int_{E^c} m_H(A'+B'+H-x)\, dm(x)\\
 &\leq \int_{E^c} m_H(A''+B''-x)\, dm(x)\\
 &= m(A''+B'').
\end{align*}
Thus $m(A'+B'+H)\leq m(A'+B')$, so $A'+B'+H\sim A'+B'$, as desired.

\smallskip

To show that $A'+H\sim A'$, suppose otherwise to get a contradiction.  Then $m(A'+H)>m(A')$, so $m((A'+H)+B')=m(A'+B')<m(A'+H)+m(B')$, and Proposition \ref{Ktopinequality} implies $m(H(A'+H+B'))>0$.  Since $A'+H+B'\sim A+B$, we then have $m(H(A+B))>0$ contradicting the hypothesis.   Similarly $B'+H\sim B'$.  \end{proof}

\begin{lemma}\label{specialtogeneral}
If Theorem \textup{\ref{main}} holds under the additional assumption that $A+B$ is measurable and $H(A+B)=\{0\}$ then  Theorem \textup{\ref{main}} holds in general.
\end{lemma}

\begin{proof}   By Lemma \ref{measurability} we may assume $A+B$ is measurable.  As in the proof of Lemma \ref{measurability}, we may assume that $(A,B)$ is irreducible and nonextendible.   Write $K:=H(A+B)$.  We first dispense with the case $m(K)>0$.  In that case, Lemma \ref{easy} implies $(A,B)$ satisfies conclusion (P) or (QP) of Theorem \ref{main}.  Now assume $m(K)=0$.  Let $A'\subseteq A$ and  $B'\subseteq B$ be as in the conclusion of Lemma \ref{H(A+B)}, so that $A'+K\sim A'$, $B'+K\sim B'$, and $A'+B'+K\sim A'+B'\sim A+B$.  Observe that $G/K$ is infinite, since $m(K)=0$.  For the remainder of the proof, write $G'$ for $G/K$, and write $m'$ for $m_{G'}$.

\begin{claim}
Viewing $A'+B'+K$ as a subset of $G'$, the pair $(A'+K,B'+K)$ is an irreducible, nonextendible, sur-critical pair for $G'$, and $H(A'+B'+K)=\{0_{G'}\}$.  In particular $m'(H(A'+B'+K))=0$, since $G'$ is infinite.
\end{claim}

\noindent\textit{Proof of Claim.}  We first show $H(A'+B'+K)=\{0_{G'}\}$ by contradiction.  Supposing otherwise, there exists $z\neq 0_{G'}\in G'$ such that \begin{align}\label{triangle1}
m'((A'+B'+K)\triangle(A'+B'+K+z))=0.
   \end{align} If $t\in G$ is such that $t+K=z$, then $t\notin K$, and by Proposition \ref{Weil} and the choice of $A'$ and $B'$,
   \begin{align*}
m_G((A+B)\triangle (A+B+t)) &= m_G((A'+B'+K)\triangle (A'+B'+K+t))\\
&=m'((A'+B'+K)\triangle(A'+B'+K+z))\\
&=0 ,
\end{align*} contradicting the definition of $K$.

\smallskip

If $(A'+K,B'+K)$ is extendible, then $m'(H(A'+B'+K))>0$ by Proposition \ref{Ktopinequality}, contradicting the previous paragraph.

\smallskip

If $(A'+K,B'+K)$ is reducible, there exist $A''\subseteq A'+K$ and $B''\subseteq B'+K$ such that $m'(A''+K)=m'(A'+K)$, $m'(B''+K)=m'(B'+K)$, and
\begin{align*}
m'(A''+K+B''+K)<m'(A'+K+B'+K).
  \end{align*}
Setting $C=A\cap (A''+K)$, $D=B\cap (B''+K)$, we apply Proposition \ref{Weil} and find that $m(C)=m(A)$, $m(D)=m(B)$, and $m(C+D)<m(A)+m(B)$.  This implies $(A,B)$ is reducible, contradicting our assumptions.  \hfill $\square$

\smallskip

Now we must show that if the pair $(A'+K,B'+K)$ satisfies any of the conclusions (P), (E), (K), or (QP) of Theorem \ref{main}, then so does the pair $(A,B)$.

\smallskip

If $(A'+K,B'+K)$ satisfies (P) or (E), then $H((A'+K)+(B'+K))\neq \{0_{G'}\}$, contradicting the claim above.  This case is vacuous.

\smallskip

If $(A'+K,B'+K)$ satisfies (K), let $\chi': G'\to \mathbb T$ be the homomorphism of conclusion (K), and let $I$, $J\subseteq \mathbb T$ be the corresponding intervals. If $\chi=\chi'\circ \phi$, where $\phi:G\to G'$ is the quotient map, it is easy to check (\S \ref{intervalsdef}) that $(A,B)$ satisfies $m(A)=m(\chi^{-1}(I))$, $m(B)=m(\chi^{-1}(J))$, while $A\subseteq a+\chi^{-1}(I), $ $B\subseteq b+\chi^{-1}(J)$ for some $a$, $b\in G$. We then see that $(A,B)$ satisfies (K).

\smallskip

If $(A'+K,B'+K)$ satisfies (QP) but not (P) or (E), let $A'+K=A_1'\cup A_0'$ and   $B'+K=B_1'\cup B_0'$ be decompositions of $A'+K$ and $B'+K$ satisfying (QP) with quasi-period $W$, such that $A_0'+B_0'$ is a unique expression element of $A'+B'+W+K$ in $G'/W$, and $m'(A_0'+K+B_0'+K)=m'(A_0'+K)+m'(B_0'+K)$.  Regarding $B_1'$, $B_0'$, $A_1'$, and $A_0'$ as subsets of $G$, we see that $A_1'\cup A_0'$ and $B_1'\cup B_0'$ are decompositions of $A'+K$ and $B'+K$ satisfying (QP) with quasi-period $W+K$.

\smallskip

 If $m(A_0'+K)=0$ or $m(B_0'+K)=0$, irreducibility of $(A'+K,B'+K)$ implies $m(H(A'+B'+K))>0$,  so that $m(H(A+B))>0$, contradicting our present assumptions. Thus we may assume $m(A_0'+K)>0$ and $m(B_0'+K)>0$, so that Lemma \ref{QPtoQP} applies to $(A',B')$ and $(A,B)$.   We conclude $(A,B)$ satisfies (QP).  \end{proof}

\subsection{Quasi-periodicity and complements} Throughout this subsection we make the standing assumption that $A+B$ is measurable in addition our previous assumption that $G$ is compact.

\begin{lemma}\label{complements}  If $(A,B)$ is an irreducible nonextendible sur-critical pair satisfying $m(A)>0$, $m(B)>0$, and $m(H(A+B))=0$, then
\begin{align}\label{comp1}
-B+ (A+ B)^c \sim A^c,
\end{align}
\begin{align}\label{comp2}
-A+ (A+ B)^c \sim B^c,
\end{align}
 and both $(-B,(A+B)^c)$ and $(-A,(A+B)^c)$ are irreducible nonextendible sur-critical pairs with $m(H(-B+(A+B)^c))=0$ and $m(H(-A+(A+B)^c))=0$.

\end{lemma}

Note that (\ref{comp1}) and (\ref{comp2}) imply $-B+(A+B)^c$ and $-A+(A+B)^c$ are measurable;  the proof is complicated somewhat by the hypothetical possibility that these sets are not measurable.

\smallskip

Lemma \ref{complements}  is an analogue of Lemma 2.4 of \cite{Grynkiewicz}.  As shown in \cite{Grynkiewicz},  the containment $-B+(A+B)^c\subseteq A^c$ holds unconditionally.

\begin{proof}  We first establish (\ref{comp1}).  Write $D$ for $-B+(A+B)^c$.  From the unconditional containment $D\subseteq A^c$, we have $m_*(D)\leq m(A^c)$.  If $m_*(D)=m(A^c)$, then (\ref{comp1}) holds, so we assume
\begin{align}\label{compcontradiction}
m_*(D)< m(A^c)
\end{align}  and derive a contradiction. Now
\begin{align}\label{oneminus}\begin{split}m(-B)+m((A+B)^c)&= m(B)+1-m(A+B)\\
&=m(B)+1-(m(A)+m(B))\\
&=m(A^c),\end{split}
\end{align}
so (\ref{compcontradiction}) means $m_*(-B+(A+B)^c)<m(-B)+m((A+B)^c)$, and Proposition \ref{Ktopinequality} implies $-B+(A+B)^c$ is measurable.

\smallskip

Let $E$ be the set of $y\in A^c$ such that $(y+B)\cap (A+B)^c\neq \varnothing$, so that $E= A^c \cap D$.  Let $S=A^c\setminus E$, so that $(S+B)\cap (A+B)^c=\varnothing$.  Nonextendibility of $(A,B)$ implies $m(S\setminus A)=0$. Since $S\subseteq A^c$, we have $m(S)=0$, which implies $m(E)=m(A^c)$.  From the definition of $E$, we have $E\subseteq -B+(A+B)^c$.  Combining this with the unconditional containment $-B+(A+B)^c\subseteq A^c$, we obtain (\ref{comp1}), which is actually the desired contradiction.

\smallskip

We have proved (\ref{comp1}), and (\ref{comp2}) follows by symmetry.  We finish the proof of the lemma in the following sequence of claims.

\smallskip

\noindent\textit{\textup{1}. $m(H(-A+(A+B)^c))=0$.}  Suppose otherwise to get a contradiction.  Then $m(H(B^c))>0$ by (\ref{comp2}), and consequently $m(H(B))>0$.   Now irreducibility of $(A,B)$ implies $m(H(A+B))>0$, contradicting our assumptions on $(A,B)$.

\smallskip

\noindent\textit{\textup{2}. $(-A,(A+B)^c)$ is a sur-critical pair.}  By a computation similar to (\ref{oneminus}), we get $m(-A)+m((A+B)^c)=m(B^c)$,
so $(-A,(A+B)^c)$ is a sur-critical pair by (\ref{comp2}).

\smallskip

\noindent\textit{\textup{3}. $(-A,(A+B)^c)$ is irreducible.}  Assume otherwise. By Corollary \ref{reduciblecor}, reducibility of $(-A,(A+B)^c)$ yields decompositions $A=A_1\cup A_0$ and $(A+B)^c=C_1\cup C_0$ satisfying (QP), and one of $m(A_0)=0$ or $m(C_0)=0$.  If $m(A_0)=0$, the irreducibility of $(A,B)$ implies $m(H(A+B))>0$, contradicting the hypothesis.  If $m(C_0)=0$ then $m(H((A+B)^c))>0$, and therefore $m(H(A+B))>0$, contradicting the hypothesis.

\smallskip

\noindent\textit{\textup{4}. $(-A,(A+B)^c)$ is nonextendible.}   Assume otherwise.  Then Proposition \ref{Ktopinequality} implies $m(H(-A+(A+B)^c))>0$, contradicting the previous claim to the contrary.

\smallskip

We have shown that $(-A,(A+B)^c)$ is an irreducible nonextendible sur-critical pair satisfying $m(H(-A+(A+B)^c))=0$.  Reversing the roles of $A$ and $B$ yields the corresponding description of $(-B,(A+B)^c)$.   \end{proof}

\begin{lemma}[cf.~\cite{Grynkiewicz}, Lemma 5.1]\label{QP1}  Let $(A,B)$ be a sur-critical pair with $m(A)>0$ and $m(B)>0$, and let $K$ be the subgroup of $G$ generated by $A$ \textup{(}so $K$ is compact and open\textup{)}.  If $A+B$ is aperiodic then $B$ is quasi-periodic with respect to $K$, or $B$ is contained in a coset of $K$.
\end{lemma}

\begin{proof}   Let $B= B_1\cup \dots \cup B_{l}$ be a $K$-coset decomposition of $B$.  Since $A\subseteq K$, $A+B$ is the disjoint union $\bigcup_{i=1}^l A+B_i$.  Then
\begin{align}\label{sumpartition}
m(A+B)=\sum_{i=1}^l m(A+B_i),
\end{align} while
\begin{align}\label{originalpartition}
m(A+B)=m(A)+m(B)=m(A)+\sum_{i=1}^l m(B_i).
\end{align}
Equating the right-hand sides of (\ref{sumpartition}) and (\ref{originalpartition}) we get
 \begin{align}\label{collectsplitsum}
 \sum_{i=1}^l m(A+B_{i})-m(B_{i})=m(A).
 \end{align}
 If $m(A+B_{i})< m(A)+m(B_{i})$ for all $i$, then the group $W:=K\cap \bigcap_{i=1}^l H(A+B_i)$ is open by Proposition \ref{Ktopinequality}, and $A+B$ is periodic with period $W$, contradicting the hypothesis.  Thus there is an $i$ such that $m(A+B_{i})\geq m(A)+m(B_{i})$, and (\ref{collectsplitsum}) implies $m(A+B_{j})=m(B_{j})$ for $j\neq i$. Consequently, $A\subseteq H(B_{j})$ and $m(B_j)>0$ for all $j\neq i$.  This implies $K\leqslant H(B_{j})$ and $B_j\sim B_j+K$ for all $j\neq i$.  We then have that $B$ is quasi-periodic with respect to $K$ if $l>1$, and $B$ is contained in a coset of $K$ if $l=1$. \end{proof}

\begin{corollary}\label{QP1cor}  With the hypotheses of Theorem \textup{\ref{main}}, if $m(H(A+B))=0$ and $A$ is contained in a coset of a compact open subgroup $K_0\leqslant G$, then $B$ has a quasi-periodic decomposition with respect to a compact open subgroup $K\leqslant K_0$, or $B$ is contained in a coset of $K_0$.
\end{corollary}

\begin{proof} By translation we may assume $A\subseteq K_0$.  The hypothesis $m(H(A+B))=0$ implies $A+B$ is aperiodic.  Now apply Lemma \ref{QP1} and let $K$ be the subgroup generated by $A$.  \end{proof}

In the next lemma, we assume that $A=A_1\cup A_0$ has a quasi-periodic decomposition with respect to the maximal period $K$ of $A_1$.  This is no less general than assuming that $A$ has a quasi-periodic decomposition $A=A_1'\cup A_0'$ with respect to some compact open subgroup $L$, since the maximal period $K$ of $A_1'$ must contain $L$, and then $A_0'$ is contained in a coset of $K$, since $A_0'$ is contained in a coset of $L$.

\begin{lemma}[cf.~\cite{Grynkiewicz}, Lemma 5.3]\label{QP3}  Let $(A,B)$ be an irreducible nonextendible sur-critical pair with $m(A)>0$, $m(B)>0$,  and $m(H(A+B))=0$.  Let $A=A_1\cup A_0$ be a quasi-periodic decomposition with $A_1$ periodic with maximal period $K$.  Then $B=B_1\cup B_0$, where $B_1\sim B_1+K$, $B_0$ is contained in a coset of $K$, and
\begin{enumerate}
\item[(i)]  $A_0+B_0+K$ is a unique expression element of $A+B+K$ in $G/K$,

    \smallskip

\item[(ii)]  $m(A_0+B_0)=m(A_0)+m(B_0)$,
\end{enumerate}
so $(A,B)$ satisfies conclusion \textup{(QP)} of Theorem \textup{\ref{main}}.

\end{lemma}

\begin{proof}  Observe that
\begin{align}\label{A1smaller}
m(A_1+B)<m(A+B)
\end{align}  since $m(H(A_1+B))\geq m(K)>0$, while $m(H(A+B))=0$.  Inequality (\ref{A1smaller}) implies
\begin{align}\label{1step}
\begin{split}
m(A_1+B)&\leq m(A+B)-m(A_0)\\
&=m(A)+m(B)-m(A_0)\\
&=m(A_1)+m(B),
\end{split}
\end{align}
where the first line results from $A_1+B\sim A_1+B+K$, while $A_0$ is contained in a coset of $K$, the second line is due to $m(A+B)=m(A)+m(B)$, and the third results from $A=A_1\cup A_0$ being a partition of $A$.  We also have
\begin{align}\label{inbetween}
0<m(A_0)<m(K),
\end{align} where $0<m(A_0)$ is a consequence of irreducibility of $(A,B)$ and (\ref{A1smaller}), while $m(A_0)<m(K)$ follows from the hypothesis $m(H(A+B))=0$.

\begin{claim1} Strict inequality holds in (\ref{1step}) and $H(A_1+B)=K$.
\end{claim1}

\noindent\textit{Proof of Claim \textup{1}.}  Assume $m(A_1+B)=m(A+B)-m(A_0)$ to get a contradiction. Then $m(A+B)=m(A_1+B)+m(A_0)$, so
\begin{align}\label{A0plusb0}
A+B\sim (A_1+B)\cup (A_0+b_0).
\end{align}
 for some $b_0\in B$.   Then (\ref{inbetween}) and the similarities $A_1\sim A_1+K$, $A_1+B\sim A_1+B+K$ imply that
\begin{align*}
C:=\{b\in B: (A+b)\setminus (A_1+B) \neq \varnothing\}
 \end{align*}
 is contained in a coset of $H(A_0)$.  If $m(H(A_0))=0$, then $m(C)=0$ and $(A,B)$ is reducible, contrary to the hypothesis.  If $m(H(A_0))>0$, (\ref{A0plusb0}) implies $K\cap H(A_0)\subseteq H(A+B)$, so $m(H(A+B))>0$, again contradicting our hypothesis.

\smallskip

Strict inequality in (\ref{1step}) implies $m(A_1+B)<m(A_1)+m(B)$. We now show  $H(A_1+B)=K$. Write $H$ for $H(A_1+B).$  Applying Proposition \ref{Ktopinequality} to $(A_1,B)$,  we have $A_1+B=A_1+B+H$.  Then $A_1+B+H\subseteq A+B$, so the nonextendibility of $(A,B)$ implies
\begin{align}\label{A1plusH}
A_1+H\subset_m A.
\end{align}
To see that $H=K$, first note that $A_1\sim A_1+K$, so $A_1+B+K\sim A_1+B$, and we find $K\leqslant H$ by the definition of $H$.  Now $A_1+H$ is disjoint from $A_0$, since otherwise $A_1+H=A+H$, and then $A\sim A+H$, contradicting $m(H(A+B))=0$.  Thus (\ref{A1plusH}) implies $A\setminus A_0 \sim A_1+H$.  The maximality of $K$ then implies $H\leqslant K$.  \hfill $\square$

\smallskip

Let $D:=\{b\in B: A_0+b\subseteq (A_1+B)^c\}$, and let $D= \bigcup_{i=1}^l D_i$ be a $K$-coset decomposition of $D$.

\begin{claim2}  $m(A_0+D_i)\geq m(A_0)+m(D_i)$ for some $i$.
\end{claim2}

\noindent\textit{Proof of Claim \textup{2}.}  We have $A+B= (A_1+B) \cup \bigcup_{i=1}^l A_0+D_i$.  Assuming Claim 2 fails, Proposition \ref{Ktopinequality} implies $W:=\bigcap_{i=1}^l H(A_0+D_i)$ is open.  Then  the containment $(K\cap W)\leqslant H(A+B)$ contradicts $m(H(A+B))=0$.  \hfill $\square$

\smallskip

Let $B_0$ be one of the $D_i$ satisfying $m(A_0+D_i)\geq m(A_0)+m(D_i)$, and let  $B_1=B\setminus B_0$. We aim to show that $B_1\sim B_1+K$.  Using  Claim 1, we apply Lemma \ref{overspill} to $(A_1,B)$ and find
\begin{align}\label{overspilled}
m(B_0)+m(A_1+B)\geq m(A_1)+m(B).
\end{align}
Then
\begin{align}\label{crit}
\begin{aligned} m(A+B)&\geq m(A_0+B_0) + m(A_1+B)  \\
&\geq m(A_0)+m(B_0)+m(A_1+B) \qquad && \text{(by Claim 2)}\\
&\geq m(A_0)+m(A_1)+m(B) \qquad && \text{(by (\ref{overspilled}))}\\
&= m(A)+m(B).
\end{aligned}
\end{align}
Since $m(A+B)=m(A)+m(B)$, the inequalities in (\ref{crit}) are actually equalities, and
\begin{align}\label{djunion}
A+B\sim (A_0+B_0)\cup (A_1+B),
\end{align}
 which is a disjoint union.  The hypothesis $m(H(A+B))=0$ and (\ref{djunion}) imply  $A_0+B_0+K$ is a unique expression element of $A+B+K$ in $G/K$.  Now   (\ref{djunion}) implies $B_1+A\subseteq A_1+B+K\sim A_1+B$.  The nonextendibility of $(A,B)$ then implies $B_1+K\sim B_1$.  This concludes the proof of (i). Part (ii) of the lemma follows from the first and second lines  of (\ref{crit}), as the inequalities are not strict.  \end{proof}

\begin{lemma}\label{QPsumtosummand}  Let $(A,B)$ be an irreducible nonextendible sur-critical pair for $G$ such that $m(A)>0$, $m(B)>0$, and $m(H(A+B))=0$.  If $A+B\sim D$, where $D$ is quasi-periodic with respect to a compact open subgroup $K\leqslant G$, then $(A,B)$ satisfies conclusion \textup{(QP)} of Theorem \textup{\ref{main}}. \end{lemma}

\begin{proof}  Write $C$ for $(A+B)^c$. The assumption that $A+B\sim D$, where $D$ is quasi-periodic with quasi-period $K$, implies $C\sim C'$ where $C'\subseteq C$ is quasi-periodic with respect to $K$, or $C'$ is contained in a coset of $K$.  Lemma \ref{complements} implies $(-B,C)$ is an irreducible nonextendible sur-critical pair with $m(H(-B+C))=0$, so $(-B,C')$ is such a pair with $m(H(-B+C'))=0$.  If $C'$ is contained in a coset of $K$, then Corollary \ref{QP1cor} implies $-B$, and therefore $B$, is quasi-periodic with respect to $K$, or is contained in a coset of $K$.  If $B$ is quasi-periodic with respect to $K$, then Lemma \ref{QP3} implies $(A,B)$ satisfies (QP).  If $B$ is contained in a coset of $K$, then Corollary \ref{QP1cor} implies $A$ is contained in a coset of $K$ or $A$ is quasi-periodic with respect to $K$.  Since $A+B\sim D$, which is quasi-periodic with respect to $K$, the sets $A$ and $B$ cannot both be contained in a coset of $K$, so $A$ is quasi-periodic with respect to $K$.  Again we apply Lemma \ref{QP3} and find $(A,B)$ satisfies (QP).

\smallskip

If $C'$ is not contained in a coset of $K$, then $C'$ is quasi-periodic with respect to $K$, and Lemma \ref{QP3} implies $(-B,C')$ satisfies conclusion (QP) of Theorem \ref{main}.  Then $-B$ is contained in a coset of $K$ or $-B$ is quasi-periodic with respect to $K$, and as above we conclude that $(A,B)$ satisfies (QP).  \end{proof}

\subsection{The $e$-transform}\label{etsection}  For $A$, $B\subseteq G$ and $e\in G$, form the pair $(A_e,B_e)$, where
\begin{align*}
A_e:= A\cup (B+e),\  B_e:=(A-e)\cap B.
\end{align*}
This is the Dyson $e$-transform, whose properties are well-documented.  In particular,
\begin{align}\label{et1}
A_e+B_e&\subseteq A+B, \\
\label{et2}m(A_e)+m(B_e)&=m(A)+m(B),
\end{align}
whenever $B_e$ is nonempty.  See \cite{Kneser56}, \cite{Nathansoninverse}, or \cite{TaoVu} for further exposition.

\smallskip

The proof of Proposition \ref{Kconnected} in \cite{Kneser56} relies on a sequence of pairs $(A^{(n)}, B^{(n)})$ successively derived by $e$-transform, where $\lim_{n\to \infty} m(B^{(n)})=0$ and $m(B^{(n)})>0$ for all $n$.  The next lemma facilitates the same construction for some pairs $(A,B)$ where the ambient group is disconnected.

\begin{lemma}\label{transformestimate}
If $A$, $B\subseteq G$ have positive Haar measure, $A+B$ is measurable, $m(A)+m(B)\leq 1$, and $A+B$ is aperiodic, then there is an $e\in G$ such that $B_e:=(A-e)\cap B\neq \varnothing$ and $m(B_e)\leq (1-m(B))m(B)$.
\end{lemma}

In particular, the conclusion holds when $m(H(A+B))=0$. The hypothesis ``$A+B$ is aperiodic" cannot be omitted; the conclusion fails when $B$ is a coset of a compact open subgroup $K\leqslant G$ and $A$ is a union of cosets of $K$.

\begin{proof}  Write $f(z)=m((A-z)\cap B)$, so that $f:G\to [0,1]$ is a continuous function and $\int f\, dm=m(A)m(B)$ (\cite{Kneser56}, Lemma 1). Note that $(A-z)\cap B$ is nonempty exactly when $z\in A-B$. Consider $S:=\{z:f(z)>0\}$, which is contained in $A-B$.  If $S\neq A-B$, then there is an $e\in A-B$ with $f(e)=0$, meaning $B_e\neq \varnothing$ and $m(B_e)=0$, so we are done.  If $S=A-B$, consider the average
\begin{align}\label{Saverage}
\frac{1}{m(S)}\int_S f dm =\frac{1}{m(A-B)} m(A)m(B).
\end{align}
 We estimate $m(A-B)$.  If $m(A-B)<m(A)+m(B)$, then  Corollary \ref{differencetosum} implies $A+B$ is periodic, contradicting the hypotheses.   Now $m(A-B)\geq m(A)+m(B)$, and equation (\ref{Saverage}) implies
 \begin{align*}
\frac{1}{m(S)}\int_S f dm&\leq \frac{m(A)}{m(A)+m(B)}m(B)\\
& = \Bigl(1-\frac{m(B)}{m(A)+m(B)}\Bigr)m(B)\\
&\leq (1-m(B))m(B),
 \end{align*}
where we used $m(A)+m(B)\leq 1$ in the last line.  So there is a $e\in G$ such that $0<f(e)<(1-m(B))m(B)$.  Now $0<m(B_e)\leq (1-m(B))m(B)$ by the definition of $f$.   \end{proof}

\subsection{The key lemma}     The proof of next lemma follows part of  the  argument of \cite{Grynkiewicz}, \S 6, Subcase 1.

\begin{lemma}\label{transformcase1}
Suppose $(A,B)$ is an irreducible, nonextendible, sur-critical pair having $m(A)>0$, $m(B)>0$, $m(H(A+B))=0$, and there exists $e\in G$ such that $m_*(A_e+B_e)<m(A)+m(B)$.  Then $(A,B)$ satisfies conclusion \textup{(QP)} of Theorem \textup{\ref{main}} with $K=H(A_e+B_e)$.
\end{lemma}

\begin{proof}  Without loss of generality assume $e=0$,\footnote{Write $C=A-e$, $D=B$, so applying the $e$-transform to $(C,D)$ with $e=0$ yields \\ $(C',D')=((A-e)\cup B, (A-e)\cap B)=(A_e-e,B_e)$.} so $A_e=A\cup B$, and $B_e=A\cap B$.  Since $m(A_e)+m(B_e)=m(A)+m(B)$, we apply Proposition \ref{Ktopinequality} to $(A_e,B_e)$ and find
\begin{align}\label{AeBeH}
A_e+B_e=A_e+B_e+H,
\end{align} while $H=H(A_e+B_e)$ is compact and open.
\begin{claim1}\label{C1} $(A\cap B)+H\sim A\cap B$.
\end{claim1}

\noindent\textit{Proof of Claim \textup{1}.}  If $b\in A\cap B$, then $A+b\subseteq A_e+B_e=A_e+B_e+H$, so $A+b+H\subseteq A+B$.  Nonextendibility of $(A,B)$ then implies $b+H\subset_m B$.  Likewise, if $a\in A\cap B$, then $a+B\subseteq A_e+B_e$, so $a+H\subset_m A$.   The claim follows.  \hfill $\square$

\smallskip

Partition $A$ as $A=A_0\cup A_1\cup A_2$, where
\begin{align*}
A_0&:=(A\cap(B+H))\setminus [(A\cap B)+H]\\
 A_1&:=A\cap [(A\cap B)+H]\\
A_2&:=A\setminus (A_0\cup A_1).
\end{align*}
Partition $B$ similarly, as $B_0:=(B\cap (A+H))\setminus (A\cap B)$, $B_1:=B\cap [(A\cap B)+H]$, and $B_2:=B\setminus (B_0\cup B_1)$.

\begin{claim2}\label{C2} $A_0+H= B_0+H$.
\end{claim2}

\noindent\textit{Proof of Claim \textup{2}.}  The definitions of $A_0$ and $B_0$ imply $A+H$ contains $B_0$, while $A_1+H$ and $A_2+H$ are disjoint from $B_0$. Thus $B_0\subseteq A_0+H$.   Similarly, $A_0\subseteq B_0+H$. \hfill $\square$

\smallskip

We now consider three cases.

\subsubsection*{Case \textup{1}. There exists $b\in B_2$ such that $(A+b)\not \subseteq (A_e+B_e)$.}    In this case (\ref{AeBeH}) implies that  there is a $b\in B_2$ and an $a_0\in A$ such that $a_0+( (b+H)\cap B)$ is disjoint from $A_e+B_e$.  Fix such $b$, and set $B_0=(b+H)\cap B$.
Note that $B_0=A_e\cap (b+H)$, by the definition of $B_2$.  Then Lemma \ref{overspill}, applied to $(A_e,B_e)$, implies
\begin{align}\label{beenot}
m(A_e+B_e)+m(B_0)\geq m(A_e)+m(B_e)=m(A+B).
\end{align}
The containment $A_e+B_e\subseteq A+B$,  (\ref{beenot}), and  the definition of $B_0$ imply
\begin{align}\label{almost}
A+B\sim(A_e+B_e)\cup (a_0+B_0).
\end{align}
Since $A_e+B_e$ is $H$-periodic, the set $(A_e+B_e)\cup (a_0+B_0)$ is quasi-periodic.  By (\ref{almost}) and  Lemma \ref{QPsumtosummand}, $(A,B)$ satisfies conclusion (QP) of Theorem \ref{main} with $K=H$.  We are done with Case 1.

\subsubsection*{Case \textup{2}. There exists $a\in A_2$ such that $(a+B)\not \subseteq (A_e+B_e)$.}  We argue as in Case 1, reversing the roles of $A$ and $B$.

\subsubsection*{Case \textup{3.}  For all $a\in A_2$ and $b\in B_2$, $A+b\subseteq A_e+B_e$ and $a+B\subseteq A_e+B_e$}  In this case, nonextendibility of  $(A,B)$ and the containment $A_e+B_e+H\subseteq A+B$ imply $b+H\subset_m B$ for all $b\in B_2$, and $a+H\subset_m A$ for all $a\in A_2$.

\smallskip

By Claims 1 and 2 and the definitions of the $A_i$, we now have $A=A_0\cup A_1\cup A_2$, where $A_1+H\sim A_1$, $A_2+H\sim A_2$, and $A_0\cap B=\varnothing$.  Similarly, $B=B_0\cup B_1\cup B_2$, where $B_1+H=B_1$, $B_2+H=B_2$, and $B_0\cap A=\varnothing$.  Also, $A_0+H=B_0+H$. Figure \ref{fig:QPfigure} depicts this scenario: the vertical rectangles are cosets of $H$, while contrary to the conclusion of the lemma, $A_0$ and $B_0$ are shown as occupying two cosets of $H$.  Our immediate goal is to show that $A+B$ has a quasi-periodic decomposition with quasi-period $H$. Let
\begin{align*}
A_0=\bigcup_{i=1}^l A_{0,i},\  B_0=\bigcup_{i=1}^l B_{0,i}
\end{align*}
be $H$-coset decompositions of $A_0$ and $B_0$, with $A_{0,i}+H=B_{0,i}+H=:H_i$.    Then $A_e\cap H_i=(A_{0,i}\cup B_{0,i})$, which is a disjoint union, for each $i$.  Lemma \ref{overspill}, applied to $(A_e,B_e)$, implies
\begin{align}\label{e-coset}
m(A_{0,i})+m(B_{0,i}) + m(A_e+B_e) \geq m(A_e)+m(B_e) =m(A+B),
\end{align}
for all $i$.

\begin{figure}
\includegraphics[scale=.4]{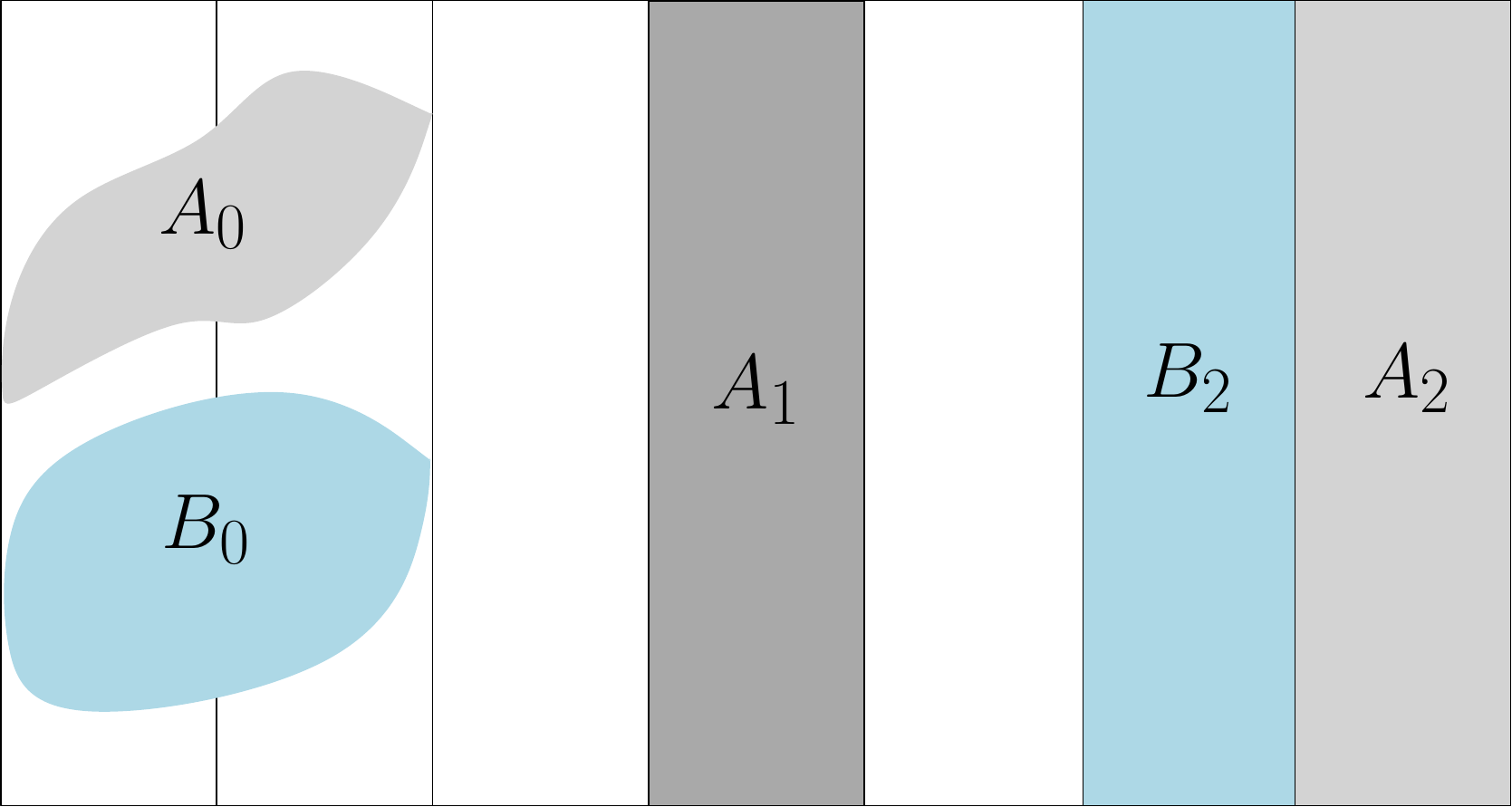}
\caption{}
\label{fig:QPfigure}
\end{figure}

\begin{claim3}  There is a pair $(i,j)$ such that $H_i+H_j$ is disjoint from $A_e+B_e$, and $m(A_{0,j}+B_{0,i})\geq m(A_{0,j})+m(B_{0,i})$ or $m(A_{0,i}+B_{0,j})\geq m(A_{0,i})+m(B_{0,j})$.
\end{claim3}

\noindent\textit{Proof of Claim \textup{3}.}  Assuming otherwise, Proposition \ref{Ktopinequality} implies the group $W:=\bigcap_{(i,j)} H(A_{0,j}+B_{0,i})\cap H(A_{0,i}+B_{0,j})$ is open, where the intersection is over all pairs $(i,j)$ such that $H_i+H_j$ is disjoint from $A_e+B_e$. Then the containment $(W\cap H)\leqslant H(A+B)$ contradicts $m(H(A+B))=0$.    \hfill $\square$

\smallskip

Assume, to get a contradiction, that $A+B$ is not quasi-periodic with quasi-period $H$. Let $(i,j)$ be a pair satisfying the conclusion of Claim 3, and let $(i',j')$ be another pair such that the sets
\begin{align*}
(A_{0,i}+B_{0,j})\cup (A_{0,j}+B_{0,i}) \text{ and } (A_{0,i'}+B_{0,j'})\cup (A_{0,j'}+B_{0,i'})
\end{align*}
 occupy distinct cosets of $H$ and are disjoint from $A_e+B_e$.  Writing
\begin{align*}
M&:= \max\{m(A_{0,i}),m(A_{0,j}),m(B_{0,i}),m(B_{0,j})\}\\
M'&:=\max\{m(A_{0,i'}),m(A_{0,j'}),m(B_{0,i'}),m(B_{0,j'})\},
\end{align*}
and $S$ for $\{i,i',j,j'\}$, we estimate $m(A+B)$ from below by
\begin{align}\label{e-estimate}
m(A+B)\geq M+M'+m(A_e+B_e).
\end{align}
Combining (\ref{e-estimate}) with (\ref{e-coset}) yields
\begin{align}\label{comparison}
M+M'\leq m(A_{0,s})+m(B_{0,s}) \text{ for all } s\in S,
\end{align}
 so in particular, $M+M' \leq \min_{s\in S} \{m(A_{0,s})+m(B_{0,s})\}$.  Assume, without loss of generality, that $M\leq M'$.  Since
\begin{align*}
\min_{s\in S} \{m(A_{0,s})+m(B_{0,s})\} \leq \min_{s\in S} \min\{m(A_{0,s}),m(B_{0,s})\}+M',
\end{align*} (\ref{comparison}) implies
\begin{align}\label{minsmin}
M=\min_{s\in S} \min\{m(A_{0,s}),m(B_{0,s})\}.
\end{align}
From the definition of $M$, (\ref{minsmin}) implies
\begin{align}\label{mmmm}
\begin{aligned}
M&=m(A_{0,i})=m(A_{0,j})=m(B_{0,i})=m(B_{0,j}).
\end{aligned}
\end{align}
Combining (\ref{mmmm}) with (\ref{comparison}) we conclude $M=M'$, and
\begin{align}\label{allsame}
m(A_{0,s})=m(A_{0,t})=m(B_{0,s})=m(B_{0,t}) \text{ for all } s,t\in S.
\end{align}
Claim 3 and (\ref{allsame}) imply $m((A_{0,i}+B_{0,j})\cup (A_{0,j}+B_{0,i}))\geq m(A_{0,i})+m(B_{0,i})$.  Again estimating $m(A+B)$, we have
\begin{align}\label{ijij}
\begin{aligned}
m(A+B)&\geq m(A_e+B_e)+m((A_{0,i}+B_{0,j})\cup (A_{0,j}+B_{0,i}))\\
&\qquad +m((A_{0,i'}+B_{0,j'})\cup (A_{0,j'}+B_{0,i'}))\\
& > m(A_e+B_e)+m(A_{0,i})+m(B_{0,i})\\
&\geq m(A)+m(B),
\end{aligned}
\end{align}
where the last line follows from (\ref{e-coset}). The strict inequality in (\ref{ijij}) is the desired contradiction.  This concludes the analysis of Case 3 and the proof of the lemma. \end{proof}

\subsection{Intersections and measure}

\begin{lemma}\label{intersections}  If $A$, $B\subseteq G$ have $m(A)>0$ and $m(B)>0$, there are sets $A'\subseteq A$, $B'\subseteq B$ with $m(A')=m(A)$ and $m(B')=m(B)$, such that $m((A'-t)\cap B')>0$ whenever $(A'-t)\cap B'$ is nonempty.
\end{lemma}

\begin{proof}  By Theorem A of \cite{Mueller65}, there is a sequence of neighborhoods $U_n$ of the identity $0\in G$ such that for $m$-almost all $a\in A$ and $b\in B$,
\begin{align}\label{PoD}
\lim_{n\to \infty} \frac{m(A\cap (U_n+a) )}{m(U_n)}=1 \text{ and }  \lim_{n\to \infty} \frac{m(B\cap (U_n+b) )}{m(U_n)} =1.
\end{align}
Let $A'\subseteq A$ and  $B'\subseteq B$ be the sets of points in $A$ and $B$, respectively, satisfying (\ref{PoD}), so that $m(A')=m(A)$ and $m(B')=m(B)$.  If  $t\in G$ and $(A'-t)\cap B'$ is nonempty, assume without loss of generality that $0\in (A'-t)\cap B'$.  Now (\ref{PoD}) implies that for $n$ sufficiently large, \begin{align*}
m((A'-t)\cap U_n)>m(U_n)/2 \text{ and } m(B'\cap U_n)>m(U_n)/2,
\end{align*}
so $m((A'-t)\cap B'\cap U_n)>0$.  Thus $m((A'-t)\cap B')>0$. \end{proof}

\begin{remark}  Theorem A of \cite{Mueller65} is proved in \cite{Mueller62}.  We outline a way to find neighborhoods $U_n$ satisfying (\ref{PoD}):  by Proposition 2.42 of \cite{FollandACAHA} and the ensuing remarks, there is a sequence of neighborhoods $V_n$ of the identity $0\in G$ such that the functions
\begin{align*}
f_n(z):= m(A\cap (V_n+z))/m(V_n),\  g_n(z):=m(B\cap (V_n+z))/m(V_n)
\end{align*}
converge in $L^1(m)$ to the functions $1_A$ and $1_B$, respectively.  One may then choose a subsequence $\{U_n\}_{n\in \mathbb N}$ of $\{V_n\}_{n\in \mathbb N}$ so that the corresponding subsequences of $f_n$ and $g_n$ converge $m$-almost everywhere.  These $U_n$ will satisfy (\ref{PoD}).  \hfill $\blacksquare$ \end{remark}

\section{The main argument}\label{main argument}

\subsection{The sequence of $e$-transforms}\label{transforms}  As in \cite{Kneser56}, we construct a sequence of pairs by successively applying the $e$-transform (\S \ref{etsection}).

\begin{lemma}\label{esequence}
Let $(A,B)$ be an irreducible, nonextendible pair satisfying the hypotheses of Theorem \textup{\ref{main}} such that $m(H(A+B))=0$.  Then at least one of the following holds.
\begin{enumerate}
\item[(i)]  There is a sequence of pairs $(A^{(n)}, B^{(n)})$, $n=0,1,2,\dots$ such that $A^{(0)}=A$, $B^{(0)}=B$, and for all $n\geq 1$,

\smallskip

    \begin{enumerate}
    \item[(i.1)]  the pair $(A^{(n)},B^{(n)})$ is derived from $(A^{(n-1)},B^{(n-1)})$ by $e$-transform,

    \item[(i.2)]  $A^{(n)}+B^{(n)}\sim A+B$, and

    \item[(i.3)] $0<m(B^{(n)})\leq (1-m(B^{(n-1)}))m(B^{(n-1)})$.
    \end{enumerate}

    \smallskip

\item[(ii)]  $(A,B)$ satisfies conclusion \textup{(QP)} of Theorem \textup{\ref{main}}.
\end{enumerate}
\end{lemma}

\begin{proof}  Suppose $(A^{(n)},B^{(n)})$ satisfies the description in (i) for $n=0,\dots, k$.  We will show that there is a pair $(A^{(k+1)},B^{(k+1)})$ satisfying (i.1)-(i.3) with $n=k+1$, or  (ii) holds. Note that $(A^{(k)},B^{(k)})$ is not extendible, as  (i.2) implies $H(A^{(k)}+B^{(k)})=\{0\}$.

\smallskip

If $(A^{(k)},B^{(k)})$ is reducible, Corollary \ref{reduciblecor} says that $(A^{(k)},B^{(k)})$ satisfies (QP) of Theorem \ref{main}.  Since $A+B\sim A^{(k)}+B^{(k)}$, Lemma \ref{QPsumtosummand} implies $(A,B)$ satisfies conclusion (QP) of Theorem \ref{main}, and we have (ii).

\smallskip

If $(A^{(k)}, B^{(k)})$ is irreducible, apply Lemma \ref{intersections} to find $C\subseteq A^{(k)}$ and $D\subseteq B^{(k)}$ such that $m(C)+m(D)=m(A^{(k)})+m(B^{(k)})$ and $(C-e)\cap D\neq \varnothing$ implies $m((C-e)\cap D)>0$.  By Lemma \ref{transformestimate}, there is an $e\in G$ such that
\begin{align}\label{Destimate}
0<m(D_e)\leq (1-m(D))m(D)
 \end{align}
Take $B^{(k+1)}=(B^{(k)})_e$ and  $A^{(k+1)}=(B^{(k)})_e$.  If
\begin{align*}
m(A^{(k+1)}+B^{(k+1)})=m(A)+m(B)
\end{align*} we have (i.1)-(i.3) with $n=k+1$, since $B^{(k+1)}\sim D_e$.  Otherwise,
\begin{align*}
m(A^{(k+1)}+B^{(k+1)})<m(A^{(k+1)})+m(B^{(k+1)})
\end{align*} and Lemma \ref{transformcase1} implies $A^{(k)}+B^{(k)}$ ($\sim A+B$) is quasi-periodic.   Now  Lemma \ref{QPsumtosummand} implies $(A,B)$ satisfies conclusion (QP) of Theorem \ref{main}.  We conclude (ii).

\smallskip

If the construction of $(A^{(n)},B^{(n)})$ yields conclusion (ii) for some $n$, we are done.  Otherwise, we have (i).   \end{proof}

We now prove a special case of Theorem \ref{main}.

\begin{proposition}\label{specialprop}  Theorem \textup{\ref{main}} holds under the additional assumptions that $(A,B)$ is irreducible and nonextendible, $A+B$ is measurable, and $H(A+B)=\{0\}$.
\end{proposition}

\begin{proof}  If $G$ is finite then conclusion (P) of Theorem \ref{main} holds.  If $G$ is infinite, the hypotheses of Lemma \ref{esequence} are satisfied. If (ii) holds in the conclusion of Lemma \ref{esequence} then we are done.  Now we must analyze the case where the conclusion (i) holds in Lemma \ref{esequence}.  Let $(A^{(n)},B^{(n)})$ be a sequence of pairs satisfying (i) of Lemma \ref{esequence}.  We argue as in the proof of Proposition \ref{Kconnected} in \cite{Kneser56}.

\begin{claim} \label{nhoods} For every neighborhood $U$ of $0\in G$, there exists $n$ with $B^{(n)}-B^{(n)}\subseteq U$.
\end{claim}

\noindent\textit{Proof of Claim.}  Suppose not.  Then, since $B^{(n+1)}\subseteq B^{(n)}$ for all $n$, there exists a neighborhood $U$ of $0$ and elements $x_n$, $y_n\in B^{(n)}$ such that $x_n-y_n\notin U$ for all $n\in \mathbb N$. Since $G$ is compact, the sequences $x_n$, $y_n$ have limit points $x$, $y$, with $x-y\neq 0$.   Our choice of $x_n$ and $y_n$ implies
\begin{align}\label{xnyn}
A^{(n)}\subseteq (A^{(n)}+B^{(n)}-y_n)\cap (A^{(n)}+B^{(n)}-x_n).
\end{align}  Condition (i.3) in Lemma \ref{esequence} implies $\lim_{n\to \infty} m(B^{(n)})=0$, so (\ref{et2}) implies $\lim_{n\to \infty} m(A^{(n)})=m(A)+m(B)$.  Then
\begin{align*}
m((A+B)\cap (A+B+y&-x))\geq \liminf_{n\to \infty} m((A+B)\cap (A+B+y_n-x_n))\\
&= \liminf_{n\to \infty} m((A^{(n)}+B^{(n)}-y_n)\cap (A^{(n)}+B^{(n)}-x_n))\\
&\geq \lim_{n\to \infty} m(A^{(n)})\\
&= m(A)+m(B)\\
&= m(A+B),
\end{align*}
where the second line uses translation-invariance of $m$ and the third uses (\ref{xnyn}). Now $A+B+(y-x)\sim A+B$, contradicting the assumption $H(A+B)=\{0\}$.  \hfill $\square$

\smallskip

We now consider two separate cases.    Corollary 7.9 of \cite{HewittandRoss1} implies these two cases cover all possible compact abelian groups $G$.

\subsubsection*{Case \textup{1}. $G$ has open subgroups of arbitrarily small measure}

By the Claim, there is a compact open subgroup $K$ and $n\in \mathbb N$ such that $B^{(n)}-B^{(n)}\subseteq K$, while $m(A^{(n)})>m(K)$.   Then $B^{(n)}$ is contained in a coset of $K$, while $A^{(n)}$ is not. Then Corollary \ref{QP1cor} implies $A^{(n)}$ is quasi-periodic, and Lemma \ref{QP3} implies $A^{(n)}+B^{(n)}$ is quasi-periodic.  Since $A+B\sim A^{(n)}+B^{(n)}$,  Lemma \ref{QPsumtosummand} implies $(A,B)$ satisfies (QP) of Theorem \ref{main}.

\subsubsection*{Case \textup{2}. $G$ has an open connected subgroup.}

Let $G_0$ be an open connected subgroup of $G$.  By the Claim, we can choose $n$ sufficiently large that $B^{(n)}$ is contained in a  coset of $G_0$.  If $A^{(n)}$ is contained in a coset of $G_0$, then by the construction of $A^{(n)}$, $A$ and $B$ are each contained in a coset of $G_0$.  Then Proposition \ref{Kconnected} implies the existence of a continuous surjective homomorphism $\chi:G_0\to \mathbb T$ and intervals $I,J\subseteq \mathbb T$ such that $A\subseteq a+\chi^{-1}(I)$, $B\subseteq b+\chi^{-1}(J)$, for some $a$, $b\in G$, such that $m(A)=m(\chi^{-1}(I))$ and $m(B)=m(\chi^{-1}(J))$.  From this we conclude (K) in Theorem \ref{main}.
If $A^{(n)}$ is not contained in a coset of $G_0$, we argue as in Case 1 and find that $(A,B)$ satisfies (QP).  \end{proof}

\subsection{Proof of Theorem \textup{\ref{main}}.}\label{proof}  If $(A,B)$ is extendible, we have conclusion (E).  If $(A,B)$ is nonextendible and reducible, Lemma \ref{reduciblelemma2} implies $(A,B)$ satisfies conclusion (QP).  If $(A,B)$ is irreducible and nonextendible,  Lemma \ref{specialtogeneral} allows us to assume that $A+B$ is measurable and $H(A+B)=\{0\}$.  In that case, Proposition \ref{specialprop} implies the conclusion of Theorem \ref{main}.  \hfill $\square$

\section{When $G$ is not compact}\label{lc}  A structure theorem for locally compact abelian groups (\cite{HewittandRoss1}, Theorem 24.30) says that such a group $G$ is isomorphic (as a topological group) to $\mathbb R^n\times L$, where $L$ is a locally compact group containing a compact open subgroup, and $n$ is a nonnegative integer.  The study of the equation \begin{align}\label{equation}
m_*(A+B)=m(A)+m(B)
\end{align}
depends heavily on $n$.  Theorem 5 of \cite{Kneser56} says that when $n\geq 1$,
\begin{align}\label{BM}
m_*(A+B)^{1/n} \geq m(A)^{1/n}+m(B)^{1/n}
 \end{align}
 for measurable subsets $A$ and $B$ of $G$, and Theorem 6 of \cite{Kneser56} classifies the pairs for which equality holds in (\ref{BM}), generalizing Theorem 2 of \cite{HenstockMacbeath} to the case where $L$ is nontrivial.  When $n>1$ and $A$ and $B$ both have positive measure, (\ref{BM}) implies $m_*(A+B)>m(A)+m(B)$, so there are no such pairs satisfying (\ref{equation}) when $n>1$.  When $n=1$, Theorem 6 of \cite{Kneser56} says that if $A$ and $B$ have positive measure then (\ref{equation}) holds if and only if there are closed intervals $I$, $J\subseteq \mathbb R$, a compact open subgroup $K\leqslant L$, and $a$, $b\in G$ such that $A\subseteq a+(I\times K)$, $B\subseteq b+(J\times K)$, and $m(A)=m(I\times K)$, $m(B)=m(J\times K)$.

\smallskip

Our study of equation (\ref{equation}) is thus reduced to groups $G$ having a compact open subgroup.  We assume $A$ and $B$ have finite measure, since otherwise (\ref{equation}) is satisfied.  One can easily check that $A$ and $B$ have compact closures under these assumptions.  Replacing $G$ by the group generated by $A\cup B$, we can assume that $G$ is compactly generated.  A compactly generated group with a compact open subgroup is isomorphic to $\mathbb Z^d\times K$ for some nonnegative integer $d$ and some compact group $K$ (\cite{HewittandRoss1}, Theorem 9.8). For such groups $G$, it is then routine to verify that Theorem \ref{main} holds with no modification of the conclusion.  This is the content of the following corollary.

\begin{corollary}\label{generalization}  The conclusion of Theorem \textup{\ref{main}} holds under the weaker assumption that $G$ has a compact open subgroup and $m_*(A+B)=m(A)+m(B)<\infty$. \hfill $\square$
\end{corollary}

\section{Acknowledgements}  This work was conducted mostly while the author was a postdoctoral fellow in the Department of Mathematics at the University of British Columbia.  The author thanks Izabella {\L}aba and Malabika Pramanik for financial support and helpful discussions.  Discussions with Michael Bj\"orklund and Alexander Fish provided motivation for this work.

\bibliographystyle{amsplain}
\frenchspacing
\bibliography{inversetheorem}

\providecommand{\bysame}{\leavevmode\hbox to3em{\hrulefill}\thinspace}
\providecommand{\MR}{\relax\ifhmode\unskip\space\fi MR }
\providecommand{\MRhref}[2]{%
  \href{http://www.ams.org/mathscinet-getitem?mr=#1}{#2}
}
\providecommand{\href}[2]{#2}
\begin{thebibliography}{10}

\bibitem{Bilu}
Yuri Bilu, \emph{The {$(\alpha+2\beta)$}-inequality on a torus}, J. London
  Math. Soc. (2) \textbf{57} (1998), no.~3, 513--528. \MR{1659821 (99j:11009)}

\bibitem{FollandACAHA}
Gerald~B. Folland, \emph{A course in abstract harmonic analysis}, Studies in
  Advanced Mathematics, CRC Press, Boca Raton, FL, 1995. \MR{1397028
  (98c:43001)}

\bibitem{GrynkiewiczKemperman}
David~J. Grynkiewicz, \emph{Quasi-periodic decompositions and the {K}emperman
  structure theorem}, European J. Combin. \textbf{26} (2005), no.~5, 559--575.
  \MR{2126639 (2005m:05222)}

\bibitem{Grynkiewicz}
\bysame, \emph{A step beyond {K}emperman's structure theorem}, Mathematika
  \textbf{55} (2009), no.~1-2, 67--114. \MR{2573603 (2010k:11161)}

\bibitem{HaRo}
Yahya~Ould Hamidoune and {\O}ystein~J. R{\o}dseth, \emph{An inverse theorem mod
  {$p$}}, Acta Arith. \textbf{92} (2000), no.~3, 251--262. \MR{1752029
  (2001c:11114)}

\bibitem{HaSeZe}
Yahya~Ould Hamidoune, Oriol Serra, and Gilles Z{\'e}mor, \emph{On the critical
  pair theory in {$\Bbb Z/p\Bbb Z$}}, Acta Arith. \textbf{121} (2006), no.~2,
  99--115. \MR{2216136 (2007a:11145)}

\bibitem{HenstockMacbeath}
R.~Henstock and A.~M. Macbeath, \emph{On the measure of sum-sets. {I}. {T}he
  theorems of {B}runn, {M}inkowski, and {L}usternik}, Proc. London Math. Soc.
  (3) \textbf{3} (1953), 182--194. \MR{0056669 (15,109g)}

\bibitem{HewittandRoss1}
Edwin Hewitt and Kenneth~A. Ross, \emph{Abstract harmonic analysis. {V}ol.
  {I}}, second ed., Grundlehren der Mathematischen Wissenschaften [Fundamental
  Principles of Mathematical Sciences], vol. 115, Springer-Verlag, Berlin,
  1979, Structure of topological groups, integration theory, group
  representations. \MR{551496 (81k:43001)}

\bibitem{Jin10}
Renling Jin, \emph{Characterizing the structure of {$A+B$} when {$A+B$} has
  small upper {B}anach density}, J. Number Theory \textbf{130} (2010), no.~8,
  1785--1800. \MR{2651155}

\bibitem{Kemperman60}
J.~H.~B. Kemperman, \emph{On small sumsets in an abelian group}, Acta Math.
  \textbf{103} (1960), 63--88. \MR{0110747 (22 \#1615)}

\bibitem{KempermanTopological}
\bysame, \emph{On products of sets in a locally compact group}, Fund. Math.
  \textbf{56} (1964), 51--68. \MR{0202913 (34 \#2772)}

\bibitem{Kneser56}
Martin Kneser, \emph{Summenmengen in lokalkompakten abelschen {G}ruppen}, Math.
  Z. \textbf{66} (1956), 88--110. \MR{0081438 (18,403a)}

\bibitem{Macbeath}
A.~M. Macbeath, \emph{On measure of sum sets. {II}. {T}he sum-theorem for the
  torus}, Proc. Cambridge Philos. Soc. \textbf{49} (1953), 40--43. \MR{0056670
  (15,110a)}

\bibitem{MoskvinFreimanJudin}
D.~A. Moskvin, G.~A. Fre{\u\i}man, and A.~A. Judin, \emph{Inverse problems of
  additive number theory and local limit theorems for lattice random
  variables}, Number-theoretic studies in the {M}arkov spectrum and in the
  structural theory of set addition ({R}ussian), Kalinin. Gos. Univ., Moscow,
  1973, pp.~148--162. \MR{0435005 (55 \#7967)}

\bibitem{Mueller65}
Bruno~J. Mueller, \emph{Three results for locally compact groups connected with
  {H}aar measure density theorem}, Proc. Amer. Math. Soc. \textbf{16} (1965),
  1414--1416. \MR{0200385 (34 \#280)}

\bibitem{Mueller62}
Bruno M{\"u}ller, \emph{Eine {V}ersch\"arfung f\"ur {A}bsch\"atzungen von
  {S}ummenmengen in lokalkompakten {G}ruppen}, Math. Z. \textbf{78} (1962),
  199--204. \MR{0150234 (27 \#235)}

\bibitem{Nathansoninverse}
Melvyn~B. Nathanson, \emph{Additive number theory}, Graduate Texts in
  Mathematics, vol. 165, Springer-Verlag, New York, 1996, Inverse problems and
  the geometry of sumsets. \MR{1477155 (98f:11011)}

\bibitem{Raikov}
D.~Raikov, \emph{On the addition of point-sets in the sense of {S}chnirelmann},
  Rec. Math. [Mat. Sbornik] N.S. \textbf{5(47)} (1939), 425--440. \MR{0001776
  (1,296b)}

\bibitem{ReiterStegeman}
Hans Reiter and Jan~D. Stegeman, \emph{Classical harmonic analysis and locally
  compact groups}, second ed., London Mathematical Society Monographs. New
  Series, vol.~22, The Clarendon Press Oxford University Press, New York, 2000.
  \MR{1802924 (2002d:43005)}

\bibitem{Shields}
A.~Shields, \emph{Sur la mesure d'une somme vectorielle}, Fund. Math.
  \textbf{42} (1955), 57--60. \MR{0072201 (17,245f)}

\bibitem{TaoVu}
Terence Tao and Van Vu, \emph{Additive combinatorics}, Cambridge Studies in
  Advanced Mathematics, vol. 105, Cambridge University Press, Cambridge, 2006.
  \MR{2289012 (2008a:11002)}

\bibitem{Vosper2}
A.~G. Vosper, \emph{Addendum to ``{T}he critical pairs of subsets of a group of
  prime order''}, J. London Math. Soc. \textbf{31} (1956), 280--282.
  \MR{0078368 (17,1182g)}

\bibitem{Vosper1}
\bysame, \emph{The critical pairs of subsets of a group of prime order}, J.
  London Math. Soc. \textbf{31} (1956), 200--205. \MR{0077555 (17,1056c)}

\end{thebibliography}

\end{document}